\documentclass[11pt,reqno,a4paper]{amsart}

\usepackage{todonotes}
\usepackage{amsmath}
\usepackage{amsfonts}
\usepackage{mathrsfs}  
\usepackage{amssymb}
\usepackage{amsthm}
\usepackage{hyperref}
\usepackage{graphicx}
\usepackage[noadjust]{cite}
\usepackage{caption}
\usepackage{dutchcal}
\usepackage{bbm}
\usepackage{tikz}
\usetikzlibrary{decorations.pathreplacing}
\usepackage[T1]{fontenc}
\usepackage[utf8]{inputenc}
\usepackage{mathtools}

\usepackage[top=3.35cm, bottom=3.35cm, left=2.5cm, right=2.5cm, headsep=0.2in]{geometry}

\usepackage{fancyhdr}

\DeclareMathOperator{\dist}{dist}

\DeclareMathOperator{\id}{id}
\DeclareMathOperator{\intt}{int}

\DeclareMathOperator{\vol}{vol}

\theoremstyle{plain}
\newtheorem{theorem}{Theorem}[section]
\newtheorem{definition}[theorem]{Definition}
\newtheorem{corollary}[theorem]{Corollary}
\newtheorem{lemma}[theorem]{Lemma}
\newtheorem{rem}[theorem]{Remark}
\newtheorem{prop}[theorem]{Proposition}

\newtheorem{ex}[theorem]{Example}

\newcommand{\RR} {\mathbb R}

\newcommand{\pa} {\partial}
\newcommand{\Cal} {\mathcal}

\newcommand{\beq} {\begin{equation}}
\newcommand{\eeq} {\end{equation}}

\newcommand{\ep} {\varepsilon}

\newcommand{\loc} {\operatorname{loc}}

\newcommand{\supp} {\operatorname{supp}}

\newcommand{\diam}{\operatorname{diam}}

\numberwithin{equation}{section}

\begin{document}
\title{Polyhedral billiards, eigenfunction concentration and almost periodic control}
\author[]{Mihajlo Ceki\'{c}, Bogdan Georgiev and Mayukh Mukherjee}
\date{\today}
\address{Max-Planck Institute for Mathematics, Vivatsgasse 7, 53111, Bonn, Germany}
\curraddr{Laboratoire de Math\'ematiques d’Orsay, Univ. Paris-Sud, CNRS, Universit\'e Paris-Saclay, 91405
Orsay, France}
\email{mihajlo.cekic@u-psud.fr}
\address{Fraunhofer Institute, IAIS, Schloss Birlinghoven, 53757, Sankt Augustin, Germany}
\email{bogdan.georgiev@iais.fraunhofer.de}
\address{IIT Bombay, Powai, Maharashtra 400076, India}
\email{mathmukherjee@gmail.com}
\maketitle

\begin{abstract}
	We study dynamical properties of the billiard flow on convex polyhedra away from a neighbourhood of the non-smooth part of the boundary, called ``pockets''. We prove there are only finitely many immersed {\em periodic tubes} missing the pockets and moreover establish a new quantitative estimate for the lengths of such tubes. This extends well-known results in dimension $2$. We then apply these dynamical results to prove a quantitative Laplace eigenfunction mass concentration near the pockets of convex polyhedral billiards. As a technical tool for proving our concentration results on irrational polyhedra, we establish a control-theoretic estimate on a product space with an almost-periodic boundary condition. This extends previously known control estimates for periodic boundary conditions, and seems to be of independent interest.
\end{abstract}
	\section{Introduction}\label{sec:intro}
	
	In the present paper our interests are twofold. The first line of studies is related to several dynamical properties of the billiard flow on convex polyhedra in $\RR^n$.  The second line of questions addresses an allied spectral problem motivated from quantum physics: namely, the problem of high energy eigenfunction concentration. 
	
	Concerning the dynamical perspective, there has been a lot of interest in the billiard flow on polygons -- for example, we refer to \cite{GKT, DG, BKM, CG, BZ, K} etc. For a more comprehensive overview, one might also wish to consult the extensive work \cite{KH}. As is well-known, this area of research contains questions that might seem quite simple looking at first, but turn out to be rather deep on closer investigation: an illustration is given by Galperin's example of non-periodic and not everywhere dense billiard trajectories in convex two-dimensional polygons (see \cite{G}).
	
	Moreover, when one increases the number of dimensions and studies the billiard flow on polyhedra (either convex or non-convex) further challenges arise. To get an insight into some of the issues involved in generalising reasonably well-known two-dimensional results to three dimensions, we refer the reader to  \cite{GKT}; see also \cite{B2}, which gives a three-dimensional analogue of a result by Katok (see \cite{K}) regarding the entropy of the symbolic dynamics of the billiard flow on polygons to higher dimensions.
	

	In this work we consider the billiard flow on polyhedra, particularly the interaction of the flow with the pockets of the billiard (i.e. an open neighbourhood around the singular set of the polytope - cf. Section \ref{subsec:bill_dyn_def}). Our main interest is in the finer properties of billiard dynamics away from the pockets. In this direction our results say that for convex polyhedra, the ``non-singular part'' of the phase space away from the pockets essentially decomposes into a collection of finitely many immersed ``periodic tubes''. 
	Moreover, we show that the lengths of the immersed ``periodic tubes'' satisfy a certain quantitative estimate. 
	
	We now turn to the second line of research devoted to the study of the underlying spectral problem. Let $-\Delta$ denote the Dirichlet (or Neumann, as the case might be) Laplacian on a convex polyhedron $P \subset \RR^n$. In \cite{HHM} concentration of eigenfunctions on certain flat polygonal surfaces has been investigated. The result states that given a polygon $P$ whose vertex set is denoted by $V$ and an open neighbourhood $U$ of $V$, there exists a positive constant $c(U)$ such that for any Dirichlet eigenfunction $u$ (satisfying $-\Delta u = \lambda u$) the following concentration estimate holds:
	\begin{equation}\label{eq:main-result}
		\int_U |u|^2 \geq c \int_P |u|^2.
	\end{equation}
	In other words, a certain amount of eigenfunction mass collects near the ``singular points'' of the boundary $\pa P$, or, $U$ {\em analytically controls} $P$ in a certain sense (see \cite{RT}). Clearly, this can be thought of as an ergodicity phenomenon. We study this phenomenon in further settings, namely higher dimensional polytopes. Since the Laplace eigenfunctions are eigenfunctions of the wave propagator $U(t) := e^{it\sqrt{-\Delta}}$, it is expected (at least heuristically) that the high-frequency limits $\lambda \to \infty$ should reflect the dynamics of the classical geodesic flow. In particular, when the geodesic flow is ergodic, one should expect the eigenfunctions to diffuse in phase space, which gives an intuitive feeling as to why estimates like (\ref{eq:main-result}) should be expected to hold (see also \cite{Z}). 
	However, note that a typical neighbourhood of the singular points of $\pa P$ is not large enough to geometrically control $P$ (in the sense of \cite{BLR, BZ, BZ1}), and special properties of the billiard flow on polyhedra need to be used. For other studies of eigenfunctions on polygons see \cite{M, MR}; see also \cite{BS, R} for general overviews of a somewhat different flavour. 
	
	For our main result, consider a convex polyhedron $P\subset \mathbb{R}^n$ and write $\mathcal{S}$ for the \emph{singular set} of the boundary $\partial P$. By definition $\mathcal{S} \subset \partial P$ is the union of faces of dimension $\leq n - 2$, or in other words, the $(n-2)$-skeleton of $\partial P$. This brings us to our main result
	
	\begin{theorem}\label{thm:eigenf_conc}
    	Given a neighbourhood $U$ of $\mathcal{S}$ inside the polyhedron $P$, there exists a positive constant $c = c(U)$ such that for any $L^2$-normalised Dirichlet (or Neumann) eigenfunction $u$, we have that %
    	\beq\label{ineq:eigen_conc}
    	\int_U |u|^2 \geq c.
    	\eeq
    \end{theorem}
    
	Again, we observe that the claim of Theorem \ref{thm:eigenf_conc} holds for all eigenfunctions, as opposed to ``almost all'' eigenfunctions (i.e. a density $1$ subsequence) as usually established in quantum ergodicity statements.
	
	Before getting into more details, we outline the strategy presented in \cite{HHM} that resolves the two-dimensional case $n = 2$. The idea is based upon the following two special properties of the billiard flow on polygonal domains:
	\vspace{3pt}
	\begin{equation} \label{eq:P1}
	    \begin{minipage}{0.85\textwidth}
	        \emph{Property 1.} Every billiard trajectory which avoids the neighbourhood $U$ is periodic (see \cite{GKT}). Such periodic trajectories come in $1$-parameter families and form ``cylinders'' of parallel trajectories in $P \setminus U$.
	    \end{minipage}\tag{P1}
	\end{equation}
	\begin{equation}\label{eq:P2}
	    \begin{minipage}{0.85\textwidth}
	        \emph{Property 2}. Furthermore, there are only finitely many such cylinders (see \cite{H, DG}).
	    \end{minipage}\tag{P2}
	\vspace{3pt}
	\end{equation}
	
	Now, let us assume to the contrary that a sequence of eigenfunctions concentrates away from $U$. The two dynamical properties mentioned above, together with propagation results for eigenfunctions, imply that the corresponding semiclassical measure on the sphere bundle must concentrate along the families of periodic geodesics that sweep out such cylinders. This is ruled out using an argument from \cite{M}, which in turn is based on an estimate by Burq and Zworski (see \cite{BZ, BZ1}). This concludes the discussion and proves the concentration estimate over $U$.
	
	
	\subsection{A brief outline of our results}
	
	We first discuss the dynamical side of our results. Our main innovations in this direction will be to discover suitable replacements and generalisations of Properties \eqref{eq:P1} and \eqref{eq:P2} to higher dimensional polyhedra. 
	
 It is known (see \cite{GKT}) that a billiard trajectory which avoids a neighbourhood of the singular points need not be periodic itself, but it is contained in an immersed tube that ``closes up'', with a suitable cross-section. Such a tube has the property that (a certain iterate of) the Poincar\'{e} return map on the cross-section of the tube is a rotation. This can be used as a substitute for Property \eqref{eq:P1} above, and we will call such tubes ``periodic tubes'' with slight abuse of language (see Definition \ref{def:periodic_tube} below and the explanations that follow).
	
	As we will demonstrate below, Property \eqref{eq:P2} will also generalise in our setting, in the sense that the number of such immersed periodic tubes missing a neighbourhood of $\mathcal{S}$ is also finite, which we call the ``finite tube condition''. Moreover, we further strengthen this claim, as we obtain a quantitative estimate for the lengths of such tubes.  A novelty here is that one needs to quantify the formation of singular points on the boundary of these tubes, which involves a certain dynamical invariant of the associated rotation.
	
	However, we will show that a much weaker property than the finite tube condition is sufficient for the application to eigenfunctions, and we call this Property \eqref{eq:P2'} (see Section \ref{subsec:P2'} below). It roughly says that for each periodic tube, the flow-out of a neighbourhood of the singular set contains a conical vicinity of the tube direction, near the boundary of the tube. We emphasise that despite the fact that the finite tube condition \emph{could} be used in our proofs, we regard the Property \eqref{eq:P2'} as a less strict and more natural condition, which may lead to further applications of our analytic argument.
	
	
    
    From the analytical standpoint, we prove that the Burq-Zworski type control result used by \cite{HHM} extends easily and applies to the case when $P$ is \emph{rational}. However, in the case of an irrational polyhedron, further complications arise as 
    our periodic tubes consist of billiard trajectories that are no longer individually periodic, but are ``almost periodic'', except for one ``central'' trajectory. This motivates us to introduce the new notion of an \emph{almost periodic boundary condition} on a cylinder. Our main contribution in this analytical part is proving a version of a control result due to Burq and Zworski (see \cite[Proposition 6.1]{BZ1}) that holds in higher dimensions for almost periodic boundary conditions. This is enough to address the case of irrational polyhedra and could potentially have other applications. We mention in passing that our methods also yield a generalisation of \cite[Theorem 2]{M}, see Theorem \ref{thm:conc_part_rect} below. This basically says that given any immersed (almost) periodic tube $T$ in $P$, no eigenfunction can concentrate in $T$ and away from $\pa T$.
    
      Finally, we mention that it is possible to state a more general version of Theorem \ref{thm:eigenf_conc} in a somewhat more abstract setting, as follows. Take a convex domain $\Omega \subset \RR^n$ with piecewise-smooth boundary, and which is such that $\pa \Omega$ consists of a finite disjoint union of convex $(n - 1)$-dimensional polyhedra $\cup_{i = 1}^k \Cal{F_i}$, which are called the ``faces'' of $\Omega$, and the set $\Cal{S} := \pa \Omega \setminus \cup_{i = 1}^k \Cal{\intt F_i} $. A variant of the standard billiard flow can be introduced on $\Omega$ which is the usual billiard flow with the stipulation that the billiard particle is stopped when it hits $\Cal{S}$. Then it can be checked that the proof of Theorem \ref{thm:eigenf_conc} goes through verbatim to give us the following corollary:
    \begin{corollary}\label{cor:abs}
    Let $U$ be any neighbourhood of $\Cal{S}$ inside $\Omega$. Then, there exists a positive constant $c = c(U)$ such that for any $L^2$-normalised Dirichlet (or Neumann) eigenfunction $u$, we have that 
    $$
    \int_U |u|^2 \geq c.
    $$
    \end{corollary}
    As an application of this corollary, one can say that if one has a cube with smoothened edges (eg., a die), then a certain amount of $L^2$-mass collects at any neighbourhood of the smoothened edges. As another application, consider the $3$-dimensional equivalent of the Bunimovich stadium, that is, a rectangular parallelepiped, with two topological hemispheres fitted smoothly at the ends. The above corollary will dictate that the $L^2$-mass of eigenfunctions cannot totally concentrate away from the ``wings'' (hemispheres) of the stadium.
    
     The main upshot is that polyhedra are in no way privileged objects, and the structure of the ``singularity'' of the $(n - 2)$-skeleton does not really play a special role. However, for aesthetic reasons, we prefer to state our main result in a more intuitive and geometrically appealing setting, as in Theorem \ref{thm:eigenf_conc}.

	\subsection{Open questions and further work}
	There are quite a few interesting questions left to consider. 
	It would be interesting to speculate if any of the dynamical results stated here have analogues for polyhedra which are not necessarily convex. 
	As a starting point, one can check whether statements like Theorem \ref{th:general-dynamics} work for not so badly non-convex objects, like a non-convex polyhedron which is formed from two convex polyhedra attached at a common face. From the analytical point of view, an exciting question is to determine the dependence of the constant $c = c(\varepsilon)$ in Theorem \ref{thm:eigenf_conc} in the case when $U$ is an $\varepsilon$-neighbourhood of the singular set, i.e. to quantify the estimates in \cite{HHM} and this paper, at least asymptotically as $\varepsilon \searrow 0$.
		
	\subsection{Structure of the paper}
	In Section \ref{sec:preliminaries} we introduce and study the new notion of functions satisfying an \emph{almost periodic boundary condition} on a cylinder. We also recall the theory of almost periodic functions and study pseudo-differential operators on a mapping torus.
	
	As mentioned before, our paper splits naturally into a dynamical and an analytical part. Section \ref{sec:billiard_dynamics} addresses some of the main dynamical components of our paper: we prove the finite tube condition (see Theorem \ref{thm:fintubecond}) and the Property \eqref{eq:P2'} (see Lemma \ref{lemma:P2'}). In Section \ref{sec:lengthsestimate}, we prove the other main dynamical result: a qualitative bound on lengths of  periodic tubes missing a neighbourhood of the singular set (see Theorem \ref{thm:partition}). Then, in Section \ref{sec:control_theory}, we discuss the analytical results required to prove Theorem \ref{thm:eigenf_conc}, concerning control-theoretic estimates with an almost periodic boundary condition (see Theorem \ref{thm:cont_per_cond}). In Section \ref{sec:mainthm} we prove Theorem \ref{thm:eigenf_conc} first in the (easier) case of rational polyhedra and then go on to prove it in full generality. 
	
	In the appendices, for the convenience of the reader and to make this work self-contained, we give important details about billiard dynamics and control theory. In Appendix \ref{app:A} we discuss theorems about billiard dynamics proved in \cite{GKT} and present proofs that work in any dimension. In Appendix \ref{app:B} we present a proof of a result of N. Burq in control theory on a product space with a periodic boundary condition.
	
	\section{Preliminaries}\label{sec:preliminaries}
	
	In the first part of this section, we discuss and recall the theory of \emph{almost periodic functions} with values in a Banach space. We then introduce the notion of an \emph{admissible isometry} and explain how admissible isometries give rise to \emph{almost periodic boundary conditions}. In fact, we show that \emph{every} isometry is admissible. In the last part, we discuss the theory of pseudo-differential operators on a mapping torus.
	
	\subsection{Almost periodic functions}\label{subsec:almostperiodic}
	
	We introduce the theory of almost periodic functions, as developed by H. Bohr in 1920s and later generalised by others. We will follow mostly the first two chapters of \cite{LZ82}. For this purpose, let $X$ be a Banach space with norm $\lVert{\cdot}\rVert$. We will say a number $\tau \in \mathbb{R}$ is an $\varepsilon$-\emph{almost period} of $f: \mathbb{R} \to X$ if
    \begin{align}\label{eq:almostperiod'}
        \sup_{t \in \mathbb{R}} \lVert{f(t + \tau) - f(t)}\rVert \leq \varepsilon.
    \end{align}
    We also say that a subset $E \subset\mathbb{R}$ is \emph{relatively dense} if there is an $l > 0$ such that for any $\alpha \in \mathbb{R}$, the interval $(\alpha, \alpha + l) \subset \mathbb{R}$ of length $l$ contains an element of $E$. 

    We start with a basic definition:
    \begin{definition}
        A continuous function $f: \mathbb{R} \to X$ is called \emph{almost periodic} if for every $\varepsilon > 0$, 
        there is an $l = l(\varepsilon) > 0$ such that for each $\alpha \in \mathbb{R}$, the interval $(\alpha, \alpha + l) \subset \mathbb{R}$ contains a number $\tau = \tau(\varepsilon)$ such that \eqref{eq:almostperiod'} holds.
    \end{definition}

    An immediate observation is that if $f$ is periodic, then it is almost periodic. 
    A simple example of a non-periodic, but almost periodic function is given by $f(t) = \sin t + \sin (\sqrt{2} t)$. Also, any almost periodic function is uniformly continuous \cite[Chapter 1]{LZ82}. Another equivalent definition is due to S. Bochner and says that (cf. \cite[p. 4]{LZ82})
    \begin{definition}
    Let $f: \mathbb{R} \to X$ be continuous. For $h \in \mathbb{R}$, we define $f^h(t) := f(t + h)$. Then $f$ is almost periodic if and only if the family of functions $\{f^h \mid h \in \mathbb{R}\}$ is compact in the topology of uniform convergence on $\mathbb{R}$.
    \end{definition}

    Next, we discuss expansion into trigonometric polynomials, i.e. harmonic analysis, similarly to the case of periodic functions. The fundamental theorem in this area is the \emph{Approximation Theorem} \cite[p. 17]{LZ82} which says that every almost periodic $f$ is a uniform limit of sums of trigonometric polynomials. In other words, for every $\varepsilon > 0$, there is a sum $\sum_{k = 1}^{n_\varepsilon} e^{i\lambda_k t} a_k (\varepsilon)$ that is $\varepsilon$-close to $f(t)$ in the uniform norm, where $a_k(\varepsilon) \in X$ and $\{\lambda_k\}_{k = 1}^\infty \subset \mathbb{R}$ is a countable set of exponents called the \emph{spectrum of $f$}. Clearly for periodic functions on $[0, 1]$, we may take $\lambda_0 = 0$ and $ \lambda_{2k-1} = 2k\pi, \lambda_{2k} = - 2k\pi $ for positive integer values of $k$. Several main properties of almost periodic functions can be deduced from the Approximation Theorem.

    We will denote the mean value of an almost periodic function $f$ (it may be shown that it exists, see \cite[p. 22]{LZ82})
    \[\mathcal{M}\{f\} := \lim_{T \to \infty} \frac{1}{2T} \int_{-T}^T f(t) dt.\]
    We now define the \emph{Bohr transformation} $a(\lambda; f)$ of $f$ for $\lambda \in \mathbb{R}$ as a ``mean value Fourier transform''
    \begin{align}
        a(\lambda; f) := \lim_{T \to \infty} \frac{1}{2T} \int_{-T}^T f(t) e^{-i \lambda t} dt = \mathcal{M} \{f(t) e^{-i\lambda t}\}.
    \end{align}
    The values $\lambda = \lambda_k$ in the spectrum of $f$ are then exactly the values for which $a(\lambda; f) \neq 0$. We will write formally, with no convergence implied, for $a_k := a(\lambda_k; f)$
    \begin{align}\label{eq:apfourier}
        f(t) \sim \sum_{k = 1}^\infty a_k e^{i\lambda_k t}.
    \end{align}
    However, one can show that certain Bochner-Fej\'{e}r sums converge uniformly to $f$, obtained by taking partial sums of the right hand side of \eqref{eq:apfourier} and applying suitable multipliers (see \cite[Section 2.4]{LZ82}). One may also show, using the Approximation Theorem, that the Fourier coefficients $a_k(\lambda_k; f)$ are uniquely associated to $f$. In other words, if $f$ and $g$ are two almost periodic functions, then $a(\lambda; f) = a(\lambda; g)$ for every $\lambda \in \mathbb{R}$ implies $f \equiv g$ (see \cite[p. 24]{LZ82}).

    Next, we assume that $X$ is a Hilbert space and state a Parseval-type identity, which says that if \eqref{eq:apfourier} holds, then
    \begin{align}\label{eq:parseval}
        \mathcal{M}\{\lVert{f(t)}\rVert^2\} = \sum_{k = 1}^\infty \lVert{a_k}\rVert^2 < \infty.
    \end{align}
    For a proof, see \cite[p. 31]{LZ82}. From \eqref{eq:parseval} it also follows that, as $n \to \infty$
    \[\mathcal{M}\{\lVert{f(t) - \sum_{k = 1}^n a_k e^{i\lambda_k t}}\rVert^2\} = \sum_{k = n + 1}^\infty \lVert{a_k}\rVert^2 \to 0.\]
    
    \subsection{Admissible isometries and almost periodic boundary conditions}\label{subsec:admissible}
    
    We consider a compact Riemannian manifold $(M_x, g_x)$ with Lipschitz boundary $\partial M_x$, where we have introduced the lower index notation to indicate explicitly that points on the manifold will be denoted by $x$. In the rest of the paper, we will deal with functions $u: M_x \times \mathbb{R} \to \mathbb{C}$, satisfying some invariance properties
	\begin{equation}\label{eq:invariantu}
	    u(x, t + L) = u(\varphi(x), t), \quad (x, t) \in M_x \times \mathbb{R}.
	\end{equation}
	Here $L > 0$ is a positive number (length) and $\varphi: M_x \to M_x$ is an isometry. It is clear that the invariance \eqref{eq:invariantu} can be interpreted as a boundary condition for $u$ on $M_x \times [0, L]$
	\begin{equation}\label{eq:boundarycondu}
	    u(x, L) = u(\varphi(x), 0), \quad x \in M_x.
	\end{equation}
    In order to address the desired control estimates (cf. Theorem \ref{thm:cont_per_cond} below), we first need to impose a condition on the isometry $\varphi$ that generalises the periodic case ($\varphi = \id$).
	
	We call an isometry $\varphi: M_x \to M_x$ with a corresponding induced isometry of the boundary $\varphi|_{\partial M_x}: \partial M_x \to \partial M_x$ \emph{admissible},
	if for every $\varepsilon > 0$, the set
    \begin{align}\label{eq:admissible}
        S(\varphi, \varepsilon) = \{k \in \mathbb{Z} \mid \dist(\varphi^k, \id) < \varepsilon\}.
    \end{align}
    is \emph{relatively dense} in $\mathbb{Z}$. In this definition, we include the possibility that the boundary $\partial M_x$ is empty. A set $A \subset \mathbb{Z}$ ($A \subset \mathbb{N}$) is relatively dense if there exists an $N \in \mathbb{N}$ such that every consecutive $N$ integers (positive integers) contain an element of $A$, i.e. every set of the form $\{k, k+1, \dotso, k + N - 1\}$ for $k \in \mathbb{Z}$ ($k \in \mathbb{N}$) contains an element of $A$. Here $\dist(\cdot, \cdot)$ denotes distance between mappings in $C^\infty(M_x, M_x)$. If the isometry $\varphi$ is admissible, we call the boundary condition in \eqref{eq:boundarycondu} \emph{almost periodic}. 
    

    
    
    

    Denote by $H^s(M_x)$ the Sobolev space of index $s \in \mathbb{R}$. The relation between almost periodic functions and admissible isometries is given by

    \begin{lemma}\label{lemma:admissiblealmostperiodic}
        Let $\varphi: M_x \to M_x$ be an admissible isometry and $u: M_x \times \mathbb{R} \to \mathbb{C}$ be such that $u \in C(\mathbb{R}, H^s(M_x))$ for some $s \in \mathbb{R}$, satisfying that $u(x, t + L) = u(\varphi(x), t)$ for all $(x, t)$ and some $L > 0$ fixed. Then the map
    \begin{equation}\label{eq:fdef}
        g: \mathbb{R} \ni t \mapsto u(\cdot, t) \in H^s(M_x)
    \end{equation}
    is almost periodic.
    \end{lemma}
    \begin{proof}
        Let $\varepsilon > 0$. By assumption, $S(\varphi, \varepsilon) \subset \mathbb{N}$ is relatively dense, so we may pick $N_\varepsilon \in \mathbb{N}$ such that $S(\varphi, \varepsilon)$ has non-empty intersection with any $N_\varepsilon$ consecutive positive integers.
    
        Fix now $\delta > 0$ and choose $\varepsilon > 0$ small enough such that, for $\psi: M_x \to M_x$ smooth
        \begin{equation}\label{eq:epsilondelta}
            \dist(\psi, \id) < \varepsilon \implies \lVert{u(\psi(x), t) - u(x, t)}\rVert_{H^s(M_x)} < \delta
        \end{equation}
        for all $t \in [0, L]$. By the invariance property of $u$ and since $\varphi^*$ is an isometric isomorphism on $H^s(M_x)$, this also holds for $t \in \mathbb{R}$.
    
        Define $\mathcal{I}(\delta) := S(\varphi, \varepsilon) L \subset \mathbb{R}$. By the defining properties of $S(\varphi, \varepsilon)$, any interval of length $LN_\varepsilon$ contains an element of $\mathcal{I}(\delta)$, i.e. $\mathcal{I}(\delta)$ is relatively dense in $\mathbb{R}$. Also, for any $k \in S(\varphi, \varepsilon)$
        \[\sup_{t \in \mathbb{R}} \lVert{g(t + Lk) - g(t)}\rVert_{H^s(M_x)} = \sup_{t \in \mathbb{R}} \lVert{u(\varphi^k(x), t) - u(x, t)}\rVert_{H^s(M_x)} < \delta\]
        by \eqref{eq:epsilondelta}. Thus the set of $\delta$-almost periods is relatively dense for any $\delta > 0$ and so $g$ is almost periodic.
    \end{proof}
    
    In a moment, we are going to demonstrate that all isometries are in fact admissible. Before proving the general case, we first discuss why any arbitrary rotation $R \in SO(n)$ acting on the closed unit ball $D^n \subset \RR^n$ is admissible. 
    
        We first show that for any $m \in \mathbb{N}$ the higher dimensional rotations on the $m$-torus $\mathbb{T}^m := \big(\frac{\mathbb{R}}{\mathbb{Z}}\big)^m$, generated by an $m$-tuple $(\alpha_1, \dotso, \alpha_m) \in \mathbb{R}^m$ and defined by
        \begin{align}\label{eq:ntorus}
            \varphi = R_{\alpha_1, \dotso, \alpha_m}: (x_1, \dotso, x_m) \mapsto (x_1 + \alpha_1, \dotso, x_m + \alpha_m) \mod \mathbb{Z}^m
        \end{align}
        are admissible. By ergodic theory of $\mathbb{T}^m$ (using Fourier expansions), we know that $\varphi$ is ergodic iff $1, \alpha_1, \dotso, \alpha_m$ are linearly independent (l.i.) over $\mathbb{Q}$, i.e. iff $\alpha_1, \dotso, \alpha_m$ are l.i. over $\mathbb{Z}$ taken modulo $\mathbb{Z}$. We denote $\alpha := (\alpha_1, \dotso, \alpha_m) \mod \mathbb{Z}^m$ and consider the orbit of $\alpha$
        \begin{equation}
            \mathcal{T} = \{k(\alpha_1, \dotso, \alpha_m) \mod \mathbb{Z}^m : k \in \mathbb{Z}\}.
        \end{equation}
        By Kronecker's theorem, we know that the closure $\overline{\mathcal{T}} \subset \mathbb{T}^m$ is a torus.\footnote{By using the ergodicity of rotations and Kronecker's theorem one should be able to deduce the claim about $\varphi$ directly.}
        
        We give an elementary proof that $\varphi$ is admissible, not relying on Kronecker's theorem or ergodicity of rotations. First observe that $\overline{\mathcal{T}}$ is a group under addition. Fix an $\varepsilon > 0$. By compactness, we may take elements $t_1 = k_1\alpha \mod \mathbb{Z}^m, \dotso, t_N = k_N \alpha \mod \mathbb{Z}^m \in \mathcal{T}$, where $k_i \in \mathbb{Z}$, such that they are $\varepsilon$-dense, i.e. for any $t \in \overline{\mathcal{T}}$, there is an $i$ such that $t_i$ is $\varepsilon$-close to $t$. Consider now an arbitrary $k \in \mathbb{Z}$ and an element
        \[t = k(\alpha_1, \dotso, \alpha_m) \mod \mathbb{Z}^m.\]
        Then $-t \mod \mathbb{Z}^m \in \mathcal{T}$ and there is an index $i$ such that $t_i$ is $\varepsilon$-close to $-t$. Equivalently, we may say that $(k_i + k)\alpha \mod \mathbb{Z}^m$ is $\varepsilon$-close to zero. Therefore, if we put
        \[l := 2\max\{|k_1|, \dotso, |k_N|\} + 1\]
        we see that in any $l$ consecutive integers we may find one, say $r$, such that $r \alpha \mod \mathbb{Z}^m$ is $\varepsilon$-close to zero. Thus $\varphi$ is admissible.

        Now, for a rotation $R \in SO(n)$, one may take a unitary matrix $P$ so that $P^{-1} R P = Q$ is diagonal and has eigenvalues coming in pairs $(e^{i\alpha}, e^{-i\alpha})$ for some generalised angles $\alpha \in \mathbb{R}$. By the discussion above and since the action of $Q$ is conjugate to a rotation on a torus as in \eqref{eq:ntorus}, we get that for any $\varepsilon > 0$, the set of indices $k$ such that $Q^k$ is $\varepsilon$-close to $\id$ is relatively dense. This implies that $R:D^{n} \to D^{n}$ is admissible.
    
    We now generalise the main ideas in the preceding paragraphs to give a proof which will not require a precise representation for the orbits as done above. For a more abstract point of view of topological dynamics and also for the converse claim, see \cite[Chapter 4]{GoHe55}. 
    
    \begin{prop}\label{prop:iso_adm}
        Let $\varphi: M_x \to M_x$ be an isometry. Then $\varphi$ is admissible.
    \end{prop}
	\begin{proof}
	    We will inductively show that $\varphi$ is \emph{$C^r$-admissible} for any $r \geq 0$, by which we mean that for each $\varepsilon > 0$ the set
	    \[S_{C^r}(\varphi, \varepsilon) = \{k \in \mathbb{Z} \mid \dist_{C^r}(\varphi^k, \id) < \varepsilon\}\] 
	    is relatively dense, where the metric $\dist_{C^r}(\cdot, \cdot)$ is now the $C^r$-metric (cf. \eqref{eq:admissible}). We start with the case $r = 0$.
	    
	    The claim hinges on proving that in compact, isometric systems, points are uniformly recurrent. For $x \in M$, consider the closure of both the forward and backward orbits
	    \[\mathcal{O}_x := \overline{\{f^m x \mid m \in \mathbb{Z}\}}.\]
	    By compactness, there is a set of indices $n_{1}, \dotso, n_{k}$ such that the points $f^{n_{1}}x, \dotso, f^{n_{k}}x$ are $\varepsilon$-dense in $\mathcal{O}_{x}$. Observe first that the set of indices $m \in \mathbb{Z}$ for which $f^mx$ is $\varepsilon$-close to $x$ is relatively dense. For, by the hypothesis there is an $i$ with 
	    \begin{equation}\label{eq:orbituniformlydense}
	        \varepsilon > d(f^mx, f^{n_i}x) = d(f^{m-n_i}x, x).
	    \end{equation}
	    Therefore, $\tilde{N} = \max(n_1, \dotso, n_k) + 1$ would work in the definition of relatively dense. Consider now the $\varepsilon$-ball $U = B(x, \varepsilon)$ around $x$ and its iterates and take any $m \in \mathbb{Z}$. Then by \eqref{eq:orbituniformlydense}, there is an $i$ such that $f^{m-n_i}x \in f^{m - n_i} U \cap U \neq \emptyset$.
	    
	    We come back to the main proof and let $U_i = B(x_i, \varepsilon)$, $i = 1, \dotso, N$ be a cover of $M$ by $\varepsilon$-balls. Apply now the previous paragraph to the isometry
	    \[(f, f, \dotso, f) : (M_x)^N \to (M_x)^N,\]
	    where on the left hand side $f$ appears $N$ times, and the point $(x_1, x_2, \dotso, x_N)$. Thus, there is a relatively dense set of integers $\mathcal{T} \subset \mathbb{Z}$ with the property that $U_i \cap f^k(U_i) \neq \emptyset$ for every $k \in \mathcal{T}$ and $i = 1, \dotso, N$. We claim that $\mathcal{T} \subset S_{C^0}(\varphi, 4\varepsilon)$. Let $k \in \mathcal{T}$, choose $y_i \in U_i \cap f^k(U_i)$ and take any $x \in M$. By construction there is an $i$ with $x \in U_i$, so $d(x, y_i) < 2\varepsilon$. Since $f$ is an isometry and $f^kx \in f^kU_i$, we have $d(f^kx, y_i) < 2\varepsilon$. By triangle inequality $d(x, f^kx) < 4\varepsilon$, which implies $\dist_{C^0}(\varphi^k, \id) \leq 4\varepsilon$ and proves the claim.
	    
	    Next, we prove the inductive hypothesis. Observe that $f$ induces a map $df: SM_x \to SM_x$ between unit sphere bundles. Moreover, this map is an isometry for the restriction of the Sasaki metric to $SM_x$ (see e.g. \cite[Chapter 1]{Pa99}). Thus we may apply the same argument as in the previous paragraph to prove the claim for $r = 1$. Similarly, by taking unit sphere bundles inductively, we may prove the claim for any $r$. Therefore, $\varphi$ is admissible.
	\end{proof}
	
	\subsection{PDOs on a mapping torus}\label{subsec:PDOmappingtorus}
	
	We will denote by $\mathcal{R} \in SO(n)$ a rotation in $\mathbb{R}^n$ that leaves the $x_n$-axis fixed. The rotation $\mathcal{R}$ is identified with the rotation it induces in $(x_1, \dotso, x_{n - 1})$ coordinates on $\mathbb{R}^{n - 1} \subset \mathbb{R}^n$. Also, define $\varphi(x_1, x_2, \dotso, x_n) := (\mathcal{R}^{-1}(x_1, \dotso, x_{n - 1}), x_n + L)$ for some positive $L$.  
	
	Consider a bounded open set $\Omega \subset \mathbb{R}^{n-1}$ invariant under the rotation $\mathcal{R}$ and define the mapping torus $\mathcal{C}_{\varphi} := (\Omega \times [0, L])/ \left( (x, L) \sim (\mathcal{R}(x), 0) \right)$. 
	The study of PDOs here is similar in spirit to the study of PDOs on the $n$-torus \cite[Chapter 5.3]{Zw}.
	
	We will consider symbol classes, for $m \in \mathbb{R}$
	\[S(m) := \{a \in C^\infty(\RR^{2n}) : |\partial^\beta a(x, \xi)| \leq C_\beta\langle{\xi}\rangle^m,\,\, \mathrm{for\,each\,multiindex} \,\, \beta\}.\]
	Here $C_\beta > 0$ is a positive constant. We denote $S(0)$ simply by $S$. The symbols are quantised by the formula $a(x, hD)u (x) = \mathcal{F}_h^{-1} a(x, \xi) \mathcal{F}_h u$, where
	\[\mathcal{F}_h u (\xi) = \int_{\mathbb{R}^n} e^{-i\frac{y \cdot \xi}{h}} u(y) dy\]
	is the semiclassical Fourier transform. Now a symbol $a \in S(m)$ defines a map $a(x, hD): \mathscr{S}(\mathbb{R}^n) \to \mathscr{S}(\mathbb{R}^n)$ and by duality $a(x, hD): \mathscr{S}'(\mathbb{R}^n) \to \mathscr{S}'(\mathbb{R}^n)$, where $\mathscr{S}(\mathbb{R}^n)$ are Schwartz functions. For $a \in S$, by standard theory we have $a(x, hD): L^2(\mathbb{R}^n) \to L^2(\mathbb{R}^n)$ uniformly in $h$. Our symbols will satisfy an additional invariance relation under $\varphi$, for $(x, \xi) \in \mathbb{R}^{2n}$
	\begin{equation}\label{eq:symbolinvariance}
	    a\big(x_1, \dotso, x_{n - 1}, x_n + L, \xi_1, \dotso, \xi_{n - 1}, \xi_n\big) = a\big(\mathcal{R}(x_1, \dotso, x_{n-1}), x_n, \mathcal{R}^{-1}(\xi_1, \dotso, \xi_{n-1}), \xi_n\big).
	\end{equation}
    
    From now on for simplicity we assume $a = a(\xi) \in S$ and satisfying \eqref{eq:symbolinvariance}, so $a = a \circ \mathcal{R}$. Then we have
        \begin{prop}\label{prop:PHI}
        The following properties hold for $\Phi$ a semiclassical PDO in $\mathbb{R}^{n}$ with symbol $a = a(\xi) \in S$, satisfying $a = a \circ \mathcal{R}$:
        \begin{itemize}
            \item[1.] $\varphi^*\Phi = \Phi \varphi^*$.
            \item[2.] $P \Phi = \Phi P$ for $P$ a constant coefficient differential operator.
        \end{itemize}
    \end{prop}
    \begin{proof}
        For the first item above, we have by definition and the change of coordinates $y' = \varphi(y)$
        \begin{align*}
            \mathcal{F}_h(u \circ \varphi) (\xi) &= \int_{\mathbb{R}^n} e^{- i\frac{y \cdot \xi}{h}} u \circ \varphi(y) dy = \int_{\mathbb{R}^n} e^{-i \varphi^{-1}(y') \cdot \frac{\xi}{h}} u(y') dy'\\
            &= e^{iL \frac{\xi_n}{h}} \int_{\mathbb{R}^n} e^{-i\frac{y'}{h} \cdot \mathcal{R}^{-1}(\xi)} u(y') dy' = e^{iL \frac{\xi_n}{h}} \mathcal{F}_h(u)(\mathcal{R}^{-1}(\xi)),
        \end{align*}
        where $\mathcal{F}_h$ denotes the semiclassical Fourier transform. This further implies, after a change of coordinate $\xi' = \mathcal{R}^{-1}(\xi)$ and using $a\circ \mathcal{R} = a$,
        \begin{multline*}
            (2\pi h)^{n} \Phi(\varphi^*u) (x) = \int_{\mathbb{R}^n} e^{i \frac{x \cdot \xi}{h}} \mathcal{F}_h(u \circ \varphi)(\xi) a(\xi) d\xi\\ = \int_{\mathbb{R}^n} e^{i L \frac{\xi_n'}{h}} e^{i\frac{x}{h} \cdot \mathcal{R}(\xi')} \mathcal{F}_h(u)(\xi') a(\xi') d\xi' = \int_{\mathbb{R}^n} e^{i \varphi(x) \cdot \frac{\xi'}{h}} \mathcal{F}_h(u)(\xi') a(\xi') d\xi', 
        \end{multline*}
        which is interpreted as $(2\pi h)^n \varphi^* (\Phi u) (x)$. 
        
        For the second point, simply recall that $\mathcal{F}_h (D^\alpha u) = \frac{\xi^\alpha}{h^{|\alpha|}} \mathcal{F}_h(u)$, where $D = - i\partial$ and $\alpha$ is any multiindex. The proof then follows from a straightforward computation.
    \end{proof}
    
    The first conclusion of Proposition \ref{prop:PHI} says that $a(hD) u$ is $\varphi$-invariant if $u$ is so, if we assume $a \in S$ satisfies the invariance property \eqref{eq:symbolinvariance}. In this sense, we may study the mapping properties of $a(hD)$ on $L^2(\mathcal{C}_{\varphi})$:
    
    \begin{prop}\label{prop:L^2bound}
        The symbol $a = a(\xi) \in S$ satisfying \eqref{eq:symbolinvariance} induces a map $a(hD): L^2(\mathcal{C}_{\varphi}) \to L^2(\mathcal{C}_{\varphi})$.
    \end{prop}
    \begin{proof}
        We follow the method of \cite[Theorem 5.5]{Zw}. Assume w.l.o.g. that $\Omega = \mathbb{R}^{n-1}$ and let $u \in L^2(\mathcal{C}_{\varphi})$. Then by a computation similar to the Proposition above, we obtain
        \begin{equation}\label{eq:expansion}
            a(hD) u(x) = \sum_{k \in \mathbb{Z}} A_ku(x), \quad x \in \mathbb{R}^{n-1} \times [0, L),
        \end{equation}
        where we write
        \[A_k = \mathbbm{1}_{\mathbb{R}^{n-1} \times [0, L)} (\varphi^{-k})^* a(hD) \mathbbm{1}_{\mathbb{R}^{n-1} \times [0, L)}.\]
        We use the notation $\mathbbm{1}_{S}$ for the characteristic function of a set $S$. Now we claim that for $|k| \geq 2$ \[\lVert{A_k}\rVert_{L^2(\mathcal{C_{\varphi}}) \to L^2(\mathcal{C_{\varphi}})} = O(h^\infty \langle{k}\rangle^{-\infty})\]
        as $h\to 0$, with a constant uniform in $k$. To prove this, notice that for any $N \in \mathbb{N}$ we have
        \[e^{\frac{i}{h}((\varphi^{-k})^*x - y) \cdot \xi} = h^{2N} |(\varphi^{-k})^*x - y|^{-2N} (-\Delta_\xi)^N e^{\frac{i}{h}((\varphi^{-k})^*x - y) \cdot \xi}. \]
        Using this formula, we may write $A_k = \mathbbm{1}_{\mathbb{R}^{n-1} \times [0, L)} (\varphi^{-k})^* \widetilde{A}_k \mathbbm{1}_{\mathbb{R}^{n-1} \times [0, L)}$, where
        \[\widetilde{A}_k v (x) = \frac{1}{(2\pi h)^n} \int_{\mathbb{R}^n} \int_{\mathbb{R}^n} \widetilde{a}_k(x, y, \xi) e^{\frac{i}{h} (x - y) \cdot \xi} v(y) dy d\xi.\]
        Here we introduced
        \[\widetilde{a}_k(x, y, \xi) = h^{2N} |x - y|^{-2N} \chi \circ \varphi^k(x) \chi(y) (-\Delta_\xi)^N a(\xi).\]
        Also, we write $\chi \in C^\infty$ for the cut-off such that $\chi = 1$ near $\mathbb{R}^{n-1} \times [0, L]$ and zero outside $\mathbb{R}^{n-1} \times [-L, 2L]$. Now by \cite[Theorem 4.20]{Zw} we may write $\widetilde{A}_k = b_k(x, hD)$ for a decaying symbol $b_k$ and then the boundedness of $b_k(x, hD)$ on $L^2(\mathbb{R}^n)$ gives the claim. The main result then follows from the expansion \eqref{eq:expansion}.
    \end{proof}
    
    Now by using Proposition \ref{prop:L^2bound} and using the standard theory on $\mathbb{R}^n$, we may obtain the usual properties of semiclassical measures on the mapping torus $\mathcal{C}_{\varphi}$: existence under an $L^2$ bound, properties of the support and invariance under flow if a suitable equation is satisfied.
	
	\section{Billiard dynamics on polyhedra}\label{sec:billiard_dynamics}

	In this section we discuss dynamical properties of the billiard flow on polyhedra. In the first two parts, we give basic definitions of the objects under study and revise the known results. In the third part, we prove a decomposition of non-singular directions on arbitrary polyhedra into finitely many tubes, i.e. the ``finite tube condition''. Finally, we prove a property of maximal periodic tubes in a polyhedron that generalises the Property \eqref{eq:P2} from the introduction. 
	
	\subsection{Billiard dynamics preliminaries}\label{subsec:bill_dyn_def}
	Let $P \subset \mathbb{R}^n$ be a polyhedron and let us define the singular set $\mathcal{S}$ as the union of all $(n-2)$-dimensional ``edges'' of the polyhedron - in other words, the $(n-2)$-skeleton of $\partial P$. This represents the higher dimensional analogues of what are ``corners'' in polygonal billiards. Let $\mathcal{S} \subset U \subset P$ be an open neighbourhood of the singular set. The billiard flow on $P$ is the usual one, a particle (or a point mass) travels to $P \setminus \mathcal{S}$ with unit velocity, and then on striking one of the faces of the boundary, instantaneously changes direction according to the law of light reflection and continues along the reflected line. Trajectories which strike a singular point stop right there - such trajectories are also called singular. If the above does not happen, then the motion is determined for all time. Now, we introduce some notations and definitions, where we largely follow the exposition in \cite{GKT}.

    \label{page2}
	Let us denote $\Gamma := \pa P$, and let $T\Gamma$ be the set of all unit tangent vectors with base points in $\Gamma$ and which are directed inside $P$. Now, define the regular (non-singular) part of $T\Gamma$ as $T\Gamma_1 := \{ x\in T\Gamma : \text{  the forward orbit of } x \text{  never hits }  \mathcal{S}\}$. Denote by $f$ the first return (Poincar\'e) map of the billiard flow to the set $T\Gamma$. Then $f$ and its iterates are defined and smooth everywhere except for the vectors whose billiard orbits hit $\mathcal{S}$, which is a set of measure zero. Let the polyhedron $P$ have $l$ faces denoted by $\mathcal{F}_1, \dotso, \mathcal{F}_l$. 
	Define $\Sigma^{+}_l := \{ 1, 2, ..., l\}^{\mathbb{N}}$ to represent the set of all forward strings for the symbolic dynamics of the billiard flow.

	
	Also, given a trajectory starting from $x \in T\Gamma_1$, the symbolic string for the forward orbit is given by $w(x)$, defined by $w(x)_i = j$ iff the basepoint of $f^i(x)$ lies in $\Cal{F}_j$. This gives us the symbolic encoding $\Sigma^+_P := \{ w \in \Sigma^+_l : \exists x \in T\Gamma_1 \text{  such that  } w = w(x)\}$, i.e., the set of all possible observable infinite strings. For each such possible string $w \in \Sigma^{+}_P$, we define $X(w) := \{x \in T\Gamma_1 : w(x) = w\}$. In other words, $X(w)$ represents all tangent vectors whose billiard trajectories have the same symbolic representation $w$. An arbitrary element of $X(w)$ is denoted by $x(w)$.
	
	We also take the space to make the following important distinction between {\em rational} and {\em irrational} polyhedra. As previously remarked, most of the literature focuses on dimension $n = 2$, where the definition of rational (irrational) polyhedra can simply be given in terms of the rationality (irrationality, respectively) of the angles at the vertices, namely, the polygon is called {\em rational} if all its angles are rational, otherwise it is called {\em irrational}. For reasons which will become clear in the course of this paper, many statements which we will concern ourselves with are easier to prove for rational polygons than irrational ones. 
	In dimensions $n \geq 3$, there does not seem to be a standard definition of rational/irrational polyhedra. We use the following definition (see \cite{B2, B3}):
	
	\begin{definition}\label{def:irrationalpoly}
		Let $P \subset \RR^n$ be a polyhedron and let $\rho_i$ represent the linear reflection determined by the $i^{\text{th}}$-face of $P$. Then $P$ is called rational if the group $G$ generated by the $\rho_i$ is finite, otherwise the polyhedron is called irrational.
	\end{definition}
	
	A crucial geometric difference exists between rational/irrational polyhedra, which we explain below. Let us first recall the well-known tool or method of unfolding a trajectory. Let $\gamma$ be a billiard trajectory of $x \in T\Gamma_1$. Starting from an initial point, we follow $\gamma$ until it strikes a face, say, $\mathcal{F}_1$ of $P$. Then we reflect $P$ about $\mathcal{F}_1$, and keep following $\gamma$ inside the reflected polyhedron until it strikes another face, whence we reflect the polyhedron again. Continuing this process indefinitely gives a sequence $P^\infty = P^0 := P, P^1, P^2,...., P^m,...$ of polyhedra which are skewered on the forward ray determined by $\gamma$. We call this object an \emph{infinite corridor} or an \emph{unfolding} along the ray $\gamma$. Obviously, on refolding the corridor the line folds back (immerses) to $\gamma$. 
	
	\begin{rem}\rm
	    We bring to readers' attention that in Definition \ref{def:irrationalpoly}, the reflections in faces are taken to be linear maps, whereas in the definition of a corridor, we take them as \emph{affine} maps. In the remainder of the paper, this distinction will be addressed in the notation.
	\end{rem}
	
	We record the following property related to unfolding: a symbolic encoding uniquely determines a direction.
	\begin{lemma}\label{lem:symbol_same_direction_same}
		In any polyhedron $P \subset \RR^n$, if $w(x) = w(y)$, then $x$ and $y$ are parallel vectors.
	\end{lemma}
    \begin{proof}
    	Unfold the polyhedron $P$ along trajectories determined by $x$ and $y$. Unless they are parallel, the distance between their trajectories grows linearly, and cannot be contained in the same corridor. The first time they move into different corridors, their symbolic representations $w(x)$ and $w(y)$ must also differ.
    \end{proof}
    
     We introduce the notation $D := (P \sqcup \sigma P)/\sim$ to denote the double of $P$, where $\sigma$ denotes a reflection of $\mathbb{R}^n$; $\sigma P$ is the copy of $P$ under this reflection and $\sim$ glues the boundaries of $P$ and $\sigma P$ by a pointwise identification of the corresponding faces in $\partial P$ and $\partial (\sigma P)$.
    
    This space plays the role of Euclidean surfaces with conical singularities (ESCS) as outlined in \cite{HHM}. The singular set $\mathcal{S} \subset D$ is now of codimension at most $2$ and the space $D_0 := D \setminus \mathcal{S}$ can be given a structure of an open Euclidean $n$-manifold. Observe that in certain cases, taking finite covers can further reduce the singular set. For example, in the case $n = 3$ a polyhedron $P$ might contain a vertex $v$ or an edge $e$ which are such that forming an $m$-fold cover $\tilde{P}$ around them converts them into removable singularities on $\tilde{P}$. This is the three-dimensional analogue of polygons having an angle of $\frac{2\pi}{m}$ at a vertex. In such cases, one can work with $\tilde{P}$ (and its double) instead of $P$ (and $D$), and the claim of Theorem \ref{thm:eigenf_conc} can be sharpened by replacing $\mathcal{S}$ by $\mathcal{S} \setminus \{v\}$ (or $\mathcal{S} \setminus \{e\}$, as the case may be). As an example, if $P$ is a rectangular parallelepiped, one can sharpen the statement of Theorem \ref{thm:eigenf_conc} to say that any neighbourhood of one of its vertices and the three edges emanating from it contain a certain fraction of the mass.\footnote{Note that we may write down eigenfunctions explicity as products of sines and cosines in this case, but there is no equidistribution in the high energy limit.} 

    Finally, we introduce some more notation for the properties of the period, different types of tubes, etc. Given a set $U \subset \mathbb{R}^{n - 1}$ and a local isometry $F: U \times \mathbb{R} \to D_0$ , we call $F(U \times \mathbb{R})$ an \emph{immersed tube} or just a \emph{tube}. We will often identify $F$ with its image $T:= F(U \times \mathbb{R})$. 
    We also call $U$ the \emph{cross-section} of the tube $T$. 
    A \emph{lifted tube} is the image of $T$ in the unit sphere bundle $SD_0$, determined by the unit vector in the positive direction of the tube. 
    
    Clearly, an immersed tube may be specified by a subset $Q$ of $T\Gamma_1$, consisting of parallel vectors whose base points form a convex set on one of the faces of the polyhedron. In dimension $n = 3$ in particular, the tube is {\em polygonal} ({\em elliptical}) if there is an open polygon (ellipse) $V$ such that $V \subset Q \subset \overline{V}$. Given a point $x \in Q$, the image in an unfolding of the tube $T = T(x)$ that is generated by $Q$ is denoted by $T^\infty(x) \subset P^\infty \subset \mathbb{R}^n$. Depending on the use which will be clear from the context, sometimes we simply write $T^\infty$ instead of $T^\infty(x)$. Unless otherwise stated, we will always assume $Q = X(w)$ for some $w \in \Sigma_P^+$ and we will refer to such tubes $T(x)$ or $T^\infty(x)$ as \emph{maximal tubes}, since they cannot be enlarged (see Proposition \ref{prop:maximaltube} below). We now formally define periodic tubes in our context. 
    
    \begin{definition}\label{def:periodic_tube}
    Let $U \subset \RR^{n - 1}$ be and let $F : U \times \RR \to D_0$ be a local isometry. Then an immersed tube $F(U \times \RR)$ is called \emph{periodic} if there is a positive number $L$ and a rotation $\mathcal{R}$ in $\mathbb{R}^{n}$ fixing the $\mathbb{R}$ direction, such that $F(x, t + L) = F(\mathcal{R} x, t)$ for all $x \in U$ and $t \in \mathbb{R}$. 
    \end{definition}
    We refer to $L$ as the length of the tube, which is also a period of the closed geodesic given by $F(\{x_0\} \times \mathbb{R})$, where $x_0 \in U$ is the centre of mass of $U$. We call $F(\{x_0\} \times \mathbb{R})$ the \emph{central geodesic} of $T$ and $\mathcal{R}$ the rotation \emph{associated to the tube $T$}. Note that periodicity of an immersed tube $T$, as defined above, is not the same as saying that all the parallel billiard trajectories contained in the tube are {\em individually} periodic.
    
    We gather a few basic properties of maximal periodic tubes in a proposition

    \begin{prop}\label{prop:maximaltube}
        The cross-section of a \emph{maximal} tube $T$ is convex. Every trajectory on the boundary of $T$ comes arbitrarily close to the singular set.
    \end{prop}
    \begin{proof}
        Let $\Omega$ be the cross-section of $T$. Consider two points $x, y \in \Omega$ and the unfolding $T^\infty$. By Lemma \ref{lem:symbol_same_direction_same}, trajectories determined by $x$ and $y$ hit the same face each time, so by the convexity of the faces of $P$, the trajectories determined by the segment $[x, y]$ always hit the interior of the faces.
        
        For the second claim, assume the trajectory determined by $x \in \partial \Omega$ is at a distance $\varepsilon > 0$ away from the singular set. Thus there is a neighbourhood $V$ of $x$, whose trajectories do not hit the singular set. Then the tube $T'$ with the cross-section $\Omega' := \Omega \cup V$ contradicts the maximality of $T$.
    \end{proof}
    
    If $n = 3$, we refer to the immersed periodic tube $T$ as a \emph{rational tube} if $\mathcal{R}$ is a rotation in a rational multiple of $\pi$ and as an \emph{irrational tube} otherwise. Of course, if the rotation $\mathcal{R}$ is by a rational multiple of $\pi$, then there exists some $\tilde{L}$ such that $F(x, t + \tilde{L}) = F(x, t)$. In other words, each parallel billiard trajectory contained in the tube is now individually periodic. For irrational tubes, this is clearly not the case, and only the central geodesic is periodic. 
    
    Given a periodic trajectory $\gamma: \mathbb{R} \to P$, we will distinguish between a \emph{period} and the \emph{minimal period}, the latter being the minimal $T_0 > 0$ such that $\gamma(t + T_0) = \gamma(t)$ for all $t$ and the former being any such $T_0$.
    
     \begin{rem}\rm
        Let us relate periodic tubes with almost periodic boundary conditions \eqref{eq:boundarycondu} defined in the previous section. Let $f \in C^\infty(D_0)$ and take a periodic tube $F: \Omega \times \mathbb{R} \to D_0$ of length $L$ and rotation $\mathcal{R}$. Consider the pullback $F^*u$ to $\Omega \times \mathbb{R}$. By definition, we have
        \[F^*u(x, t + L) = F^*u(\mathcal{R}x, t), \quad (x, t) \in \Omega \times \mathbb{R},\]
        which by Proposition \ref{prop:iso_adm} means that $F^*u$ satisfies the required condition.
    \end{rem}

    \subsection{A revision of known results}\label{subsec:dynamicsrevision}
    
    Here we collect the relevant preliminary results on polyhedral dynamics that will be of use later. For completeness, we have provided proofs in Appendix \ref{app:A}. The proofs are quite instructive and we encourage the reader to go through them.
    
	Consider a polyhedron $P$. As mentioned before, in higher dimensions, depending on the type of the polyhedron,  an orbit generated by $x \in T\Gamma_1$ with a periodic symbol $w(x)$ might or might not be periodic. This is in contrast to the polygon case (see Theorem \ref{thm:period_traj'}). Now, observe that if the forward closure of the orbit of $x$, called $\gamma$ henceforth, does not intersect $\mathcal{S}$, then $\gamma$ can be ``thickened'' to form a tubular neighbourhood $T$ around $\gamma$ such that each trajectory in $T$ parallel to $\gamma$ also has the symbolic representation $w(x)$. It turns out that such tubes $T$ can themselves be periodic (see below).
	
	With that in place, we have the following result proved in \cite[Theorem 5]{GKT}

    \begin{theorem} \label{th:general-dynamics}
		Let $P \subset \RR^n$ be an arbitrary convex 
		polyhedron and $w \in \Sigma^+_P$ is a periodic sequence with minimal period $k$. The following hold:
		
		\begin{enumerate}
			\item There exists $x(w)$ so that $x(w)$ is periodic with minimal period $k$.
			\item In addition, one of the following two cases holds:
			\begin{enumerate}
				\item There exists $q \geq 1$ such that all $y(w) \in X(w) \setminus x(w)$ are periodic with period $qk$ and the cross-section of the tube generated by $X(w)$ is an open polyhedron.
				\item The set $X(w)$ generates a periodic tube with a convex cross-section $\Omega \subset \mathbb{R}^{n - 1}$ and an associated isometry $\mathcal{R}_0 \in O(n - 1)$ keeping $\Omega$ invariant.\footnote{ By slightly abusing the notation, we extend here Definition \ref{def:periodic_tube} to include Poincar\'e maps in $O(n-1)$. Note that at the cost of doubling the length of the periodic tube, we may always assume this map is in $SO(n-1)$.}
			\end{enumerate}						
			\item  If $n = 2$ or $3$ and $k$ is odd, then only the first case {\rm (2).(a)} above can happen and $q = 2$.
			\item If $P$ is rational then only the first case {\rm (2).(a)} above can happen.  			
		\end{enumerate}
	\end{theorem}
	
	\begin{rem}\rm \label{rem:general-dynamics}
	    We remark that in \cite{GKT} a precise statement of the theorem above was given only in the case $n = 3$. In the case $n = 3$, one additionally has that the periodic tubes in $(2)$.(b) are either polygonal or elliptical, or in other words $\Omega \subset \mathbb{R}^2$ is either a convex polygon or a disc. They respectively correspond to rational and irrational periodic tubes.
	    
	    For a general value of $n$, the part $(2)$.(b) of the theorem needs to be modified and we have many more options for the cross-section $\Omega$ of a  periodic tube.  Moreover, $(3)$ now holds only for $n = 2, 3$, since in higher dimensions it is false that the product of three linear reflections in a general position is a linear reflection.
	\end{rem}
    
    We include a proof of Theorem \ref{th:general-dynamics} in Appendix \ref{app:A}. We now discuss a couple of explicit examples in the case $n = 3$. As an example of when $(2)$.(a) above might occur, consider a right prism whose horizontal cross-section is an equilateral triangle. Consider a trajectory which lies on a plane perpendicular to the height of the prism, and strikes an equilateral triangular cross-section exactly at the mid-point of the three sides. This is periodic with minimal period $3$, as well as any vertical translate of such a trajectory, whereas any other $y(w) \in X(w)$ is periodic with minimal period $6$. 
    
    For an example when $(2)$.(b) might occur, it is enough to consider a regular tetrahedron and the closed orbit corresponding to the word $w = (abcd)$, where $a, b, c$ and $d$ encode the faces of the tetrahedron. Then it is possible to show that there is a unique closed orbit $x(w)$, that $X(w)$ generates an elliptical, irrational, periodic tube and the nearby parallel trajectories ``come back'' rotated by an irrational angle. See \cite[Section 8]{B1} for explicit computations. Observe that although $X(w)$ generates an immersed solid torus $T$ with disc cross-section, the faces of $P$ intersect $T$ obliquely, and hence the footprint of $T$ on a face of $P$ looks like an ellipse.
    
    Theorem \ref{th:general-dynamics} shows that the result for the two-dimensional polygonal case in Theorem \ref{thm:period_traj'} does not generalise to higher dimensions in a completely straightforward way. In other words, there exist trajectories whose closure does not contain any singular point, but which are themselves not periodic - this might occur in the case of irrational polyhedra. Moreover, such trajectories are contained in periodic tubes.
    
    However, the following result says that Theorem \ref{th:general-dynamics} contains all such possible obstructions:
    \begin{theorem}\label{th:non-periodic}
        Let $P \subset \mathbb{R}^n$ be a convex polyhedron and $w \in \Sigma_P^+$ non-periodic. Let $x \in X(w)$. Then the closure of the trajectory generated by $x$ intersects the singular set $\mathcal{S}$.
    \end{theorem}
    
    When $n = 3$, the above result basically says that $X(w)$ is at best a codimension $1$ ``strip'', which includes the case that $X(w)$ consists of a single point. We have included a proof of Theorem \ref{th:non-periodic} in Appendix \ref{app:A}. As a consequence we obtain the following dichotomy:
    
    \begin{corollary} \label{lem:Dichotomy}
    	For any billiard trajectory $\gamma$ in a convex polyhedron $P$, either $\gamma$ is contained in an immersed periodic tube, or the closure of $\gamma$, $C_\gamma$, meets $\mathcal{S}$.
    \end{corollary}

    \subsection{The finite tube condition}
    
        We begin this section by introducing a condition on the finiteness of   periodic tubes missing $U_\varepsilon$, that will be relevant for our further discussion. Following \cite{HHM}, we have
        \begin{definition}\label{def:tube_cond}
    	    Let $D$ be the double of $P$ as defined above. A region $U \subset D$ is said to satisfy the \emph{finite tube condition} if there exists a finite collection of    periodic tubes $T_i$ for $i = 1, \dotso, N$ for some $N$, such that any orbit that avoids $U$ belongs to some $T_i$.
        \end{definition} 
    
        We will sometimes simply say that $D$ satisfies the \emph{finite tube condition} if every neighbourhood of $\mathcal{S}$ does so. In this section we prove the finite tube condition in full generality.

        \begin{theorem}\label{thm:fintubecond}
            Let $P$ be a convex polyhedron and let $D$ be its double. Then, for any $\varepsilon > 0$, the $\varepsilon$-neighbourhood $U_\varepsilon$ of the singular set $\mathcal{S}$ of $D$ satisfies the finite tube condition.
        \end{theorem}
        \begin{proof}
            Assume the finite tube condition is false, i.e. assume there are points $x_1, x_2, \dotso \in T\Gamma_1$ generating trajectories that stay in the future in $P \setminus U_\varepsilon$. By Theorem \ref{th:general-dynamics}, for each $i \in \mathbb{N}$, the trajectories generated by $x_i$ belong to maximal immersed  periodic tubes $T_i$, and we assume that the $T_i$ are distinct tubes.
            
            By compactness and after relabelling we may assume $x_i \to x$ as $i \to \infty$. By continuity, the trajectory generated by $x$ is neither tangent to a face nor has basepoint in the singular set, and stays in $P \setminus U_\varepsilon$ in the future. Therefore $x$ belongs to a maximal  periodic tube $T$ by Theorem \ref{th:general-dynamics}, where $T$ is the image of a local isometry $F : \Omega \times \RR \to D_0$. Note that for $i$ large enough, each $x_i$ points to a direction not parallel to $x$, by the fact that the tubes $T_i$ are distinct. 
            
            By Lemma \ref{lemma:singsetdense}, there are points $p_1, \dotso, p_N \in \partial \Omega$ that are $\varepsilon/100$-dense, such that for some $t_1, \dotso, t_N$ we have $F(p_i, t_i) \in \mathcal{S}$. By the same lemma
            , such points are uniformly recurrent: there is a $C > 0$, such that any slice $F^{-1}(\mathcal{S}) \cap \partial \Omega \times [x, x+C]$ of height $C$, contains $N$ points whose projections onto $\partial \Omega$ are $\varepsilon/100$-dense. In other words, $\partial T^\infty$ contains singular points whose projections onto the cross-section are $\varepsilon/100$-dense uniformly often.
        
            Consider now $T(x_i)$, i.e. the maximal tube generated by $x_i$. Note that such tubes contain immersed tubes with disc cross-section of radius $\varepsilon$. Taking $i \to \infty$, we obtain that $T^\infty(x_i)$ intersects $\partial T^\infty$ in tubes of increasingly large height. By the construction, such tubes must contain singular points eventually. This contradicts the fact that $T^\infty(x_i)$ do not contain singular points in the interior of their unfolding.
        \end{proof}
        
        
    \subsection{Property 2'.}\label{subsec:P2'} As is discussed in Section \ref{sec:intro}, the key to our proof are appropriate generalisations of Properties \eqref{eq:P1} and \eqref{eq:P2}, together with our main estimate Theorem \ref{thm:cont_per_cond} (see Section \ref{sec:control_theory} below). The generalised Property \eqref{eq:P1} is contained in Theorems \ref{th:general-dynamics} and \ref{th:non-periodic}, so we discuss a suitable substitute of Property \eqref{eq:P2} that we will call Property 2'. In order to state it, let $F: \Omega \times \mathbb{R} \to D_0$ be a maximal periodic tube and call $\Omega_{\varepsilon}$ the complement of the $\varepsilon$-neighbourhood of $\partial \Omega$.
    
    \vspace{-2.5mm}
    \begin{equation}\label{eq:P2'}
        \begin{minipage}{0.8\textwidth}
            \emph{Property 2'}. There exist $\delta > 0$, $l > 0$ and $\eta > 0$, such that for each $(p_0, z_0) \in \big(\Omega_{\varepsilon/2 - \eta} \setminus \Omega_{\varepsilon/2 + \eta}\big) \times \mathbb{R}$, there exists a point $(p_0, z)$ with $|z - z_0| \leq l$, and the $\delta$-neighbourhood of $(p_0, z)$ is contained in $F^{-1}(U_\varepsilon)$.
        \end{minipage}\tag{P2'}
    \end{equation}
    \vspace{0.5mm}
    
    We will prove the Property \eqref{eq:P2'} in this section. We begin by a lemma studying the formation of singularities on the boundary of a  periodic tube; to this end assume $L$ is the length of $T$ and $\mathcal{R}$ the associated rotation.
   
   \begin{lemma}\label{lemma:singsetdense}
        Denote by $\pi_1 : \mathbb{R}^{n-1} \times \mathbb{R} \to \mathbb{R}$ the projection to the first coordinate. Then the projections of singular points on $\partial T$ under $\pi_1$ are dense
        \begin{equation}\label{eq:densecond}
            \overline{\pi_1\big(F^{-1}(\mathcal{S})\big)} = \partial \Omega.
        \end{equation}
        Moreover, singular points are uniformly recurrent: for every $\varepsilon > 0$, there is an $L(\varepsilon) > 0$ depending on $T$, such that for every $t \in \mathbb{R}$
        \[\pi_1\Big(F^{-1}(\mathcal{S}) \cap \partial \Omega \times [t, t + L(\varepsilon)]\Big)\]
        is $\varepsilon$-dense in $\partial \Omega$. 
   \end{lemma}
   \begin{proof}
       By Proposition \ref{prop:maximaltube}, for every $p \in \partial \Omega$ and for every $\gamma > 0$ there is a $t \in \mathbb{R}$ such that $F(p, t)$ is $\gamma$-close to $\mathcal{S}$ (here we extended $F$ to $\overline{\Omega} \times \mathbb{R}$ as a continuous map to $D$). In fact, more is true and we can study the formation of singular points on the boundary of $T$. The idea of the proof of \eqref{eq:densecond} is to use the periodicity of the tube and of singular points. Recall that the local isometry $F$ is invariant under the map $(x, t) \mapsto (\mathcal{R}x, t + L)$. Since this is a rigid motion, the nearby singular points in an unfolding can only approach $\partial T^\infty$ arbitrarily close, if there appears a singular point on $\partial T^\infty$. 
       
       More precisely, consider an unfolding $P^\infty = P^0, P^1, P^2, \dotso$ along $T$ and denote by $(\partial P^i)^{(n-2)}$ the $(n-2)$-skeleton of the boundary of $\partial P^i$. Define
       \[\mathcal{S}_1 = \cup_{i = 0}^k (\partial P^i)^{(n-2)},\]
       for a sufficiently large $k$ to be specified later. Let us define the forward orbit of $\mathcal{S}_1$ under the rotation $\mathcal{R}$ as
       \[\mathcal{S}_\infty = \cup_{i = 0}^\infty \mathcal{R}^i \mathcal{S}_1.\]
       By the observation above and by taking $k$ larger than the period of the symbol associated to $T$, we know that $\partial \Omega \subset \overline{\pi_1(\mathcal{S}_\infty)}$. Also, observe that for any point $x \in \mathbb{R}^n$, since $\mathcal{R}$ is an isometry fixing $\partial T^\infty$
       \begin{equation}\label{eq:tubedistance}
           d(x, \partial T^\infty) = d(\mathcal{R}x, \partial T^\infty).
       \end{equation}
       Consider now $\pi_1(\mathcal{R}^{k_i}x_i) \to y \in \partial \Omega$, for some $k_i \geq 0$ and $x_i \in \mathcal{S}_1$. By compactness of $\mathcal{S}_\infty$, we may assume without loss of generality $\mathcal{R}^{k_i}x_i \to x \in \partial T^\infty$, so $\pi_1(x) = y$. By compactness of $\mathcal{S}_1$ and since by \eqref{eq:tubedistance} we have $d(x_i, \partial T^\infty) \to 0$, after re-labelling we may assume $x_i \to x' \in \partial T^\infty \cap \mathcal{S}_1$. But then since $\mathcal{R}$ an isometry and by triangle inequality we have $\mathcal{R}^{k_i} x' \to x$.
       
       The previous argument shows that
       \[\partial \Omega \subset \overline{\pi_1 \big(\mathcal{S}_\infty \cap \partial T^\infty\big)},\]
       which proves the first claim. The second claim follows from the first claim, the periodicity of $T$ and the fact that $\mathcal{R}$ is an isometry of $\partial \Omega$.
   \end{proof}
   
   Now we may prove this property in full
   
   \begin{lemma}\label{lemma:P2'}
        Property \eqref{eq:P2'} holds for any $n \geq 2$.
   \end{lemma}
   \begin{proof}
        Define $\eta = \delta = \varepsilon/6$. Consider a point $p \in \partial \Omega$ closest to $p_0$, at a distance at most $2\varepsilon/3$. By Lemma \ref{lemma:singsetdense}, we can find a $z$ with $|z - z_0| \leq L(\varepsilon/6)$, such that there is a point $q \in \partial \Omega$ that is $\varepsilon/6$-close to $p$ and $F(q, z) \in \mathcal{S}$. Therefore, $(p_0, z)$ is by triangle inequality $5\varepsilon/6$-close to $(q, z)$, and so its $\varepsilon/6$-neighbourhood is contained in $F^{-1}(U_\varepsilon)$.
   \end{proof}
   
   \begin{rem}\rm
        In other words, Property \eqref{eq:P2'} is saying that the flow-out of $U_\varepsilon$ under the billiard flow contains, for each maximal periodic tube $T$, a conical neighbourhood of the tube direction in a neighbourhood of $\partial T$.
   \end{rem}
    
	\section{A quantitative estimate on   periodic tube lengths}\label{sec:lengthsestimate}
    
    In this section we prove our main quantitative estimate on the lengths of   periodic tubes. For this, we first establish an estimate on the angle of intersecting tubes. A new feature compared to the polygonal case, is that these estimates depend on dynamical invariants of associated rotations, that we are about to define.
   
    We start with the case $n = 2$. Given an irrational rotation $\mathcal{R} \in SO(2)$, $\varepsilon > 0$ and a radius $r > 0$, we write $N = N(\mathcal{R}, \varepsilon, r) \in \mathbb{N}$ for the smallest positive integer such that for every $x$ lying on the circle of radius $r$, denoted by $S_r^1$, we have
    \[x, \mathcal{R}x, \dotso, \mathcal{R}^{N - 1}x\]
    is $\varepsilon$-dense in $S^1_r$. If the rotation $\mathcal{R}$ is rational of minimal order $o$, we set $N(\mathcal{R}, \varepsilon, r) = o$ for $\varepsilon$ sufficiently small. We specially set $N(\mathcal{R}, \varepsilon, 0) = 1$. By the scaling property of the circle we obtain $N(\mathcal{R}, \varepsilon, r) = N(\mathcal{R}, \frac{\varepsilon}{r}, 1)$.
    
    More generally, assume $\mathcal{R} \in SO(n)$. Given an $r \geq 0$, we will denote by $B_r \subset \mathbb{R}^n$ the ball of radius $r$; we define $B_0 = \{0\}$. Clearly $\mathcal{R}$ acts on $\overline{B_r}$. Write $\{\mathcal{O}_i(r): i \in I(r)\}$ for the  closed minimal\footnote{Minimal in the sense that $\mathcal{O}_i(r)$ does not contain any non-empty proper closed $\mathcal{R}$-invariant set.} orbits of the action of $\mathcal{R}$ on $\overline{B_r}$; the $\mathcal{O}_i(r)$ partition $\overline{B_r}$ into disjoint pieces.
    \begin{definition}\label{def:N}
        Given an $\varepsilon > 0$, we introduce
        \begin{equation}\label{eq:defproperty}
            N(\mathcal{R}, \varepsilon, r) = \max_{i \in I} \, \min \big\{ m \in \mathbb{N} : \forall x \in \mathcal{O}_i(r), \,\, \{x, \mathcal{R}x, \dotso, \mathcal{R}^{m - 1}x\} \,\, \mathrm{is} \,\, \varepsilon\text{-}\mathrm{dense \,\, in} \,\, \mathcal{O}_i(r)\big\}.
        \end{equation}
    \end{definition}
    
    Note that the orbits $\mathcal{O}_i(r)$ are in fact tori, as can be seen by conjugating $\mathcal{R}$ to a canonical form, which is a direct sum of rotations in $SO(2)$ with possibly some further $\pm 1$'s at the diagonal, and looking at the conjugated action on a suitable torus (cf. discussion before Proposition \ref{prop:iso_adm}). Note also that $N = N(\mathcal{R}, \varepsilon, r)$ in \eqref{eq:defproperty} is well-defined: the minimum always exists by the minimality and compactness of orbits and the maximum is well-defined since there are only finitely many orbit types (up to scaling and translation), by the diagonalisation argument above. In fact, for sufficiently small $\varepsilon > 0$, the maximum is achieved by a ``typical orbit'', i.e. the one with non-zero components corresponding to each rotation in $SO(2)$ or $\pm 1$ on the diagonal. Note that we also have the scaling property $N(\mathcal{R}, \varepsilon, r) = N(\mathcal{R}, \frac{\varepsilon}{r}, 1)$. Given a convex polyhedron $P$, we define $N(\mathcal{R}, \varepsilon) := N(\mathcal{R}, \varepsilon, \diam P)$.
    
    Intuitively, Definition \ref{def:N} and \eqref{eq:defproperty} give us a smallest positive integer $N$ such that the orbit of any point $x \in \overline{B_r}$ is $\varepsilon$-dense in the corresponding closed minimal orbit $\mathcal{O}_i(r)$. It is clear that this more general definition agrees with the one given for $n = 2$. We will need these notions when quantifying the formation of the singular set on the boundary of a  periodic tube. 
     To be more precise, consider a maximal tube $T = F(\Omega \times \mathbb{R})$ of length $L$ and associated rotation $\mathcal{R}$. We will use the number $N = N(\mathcal{R}, \varepsilon, \diam P)$ in Definition \ref{def:N} above to say that the intersection of a slice of $T$ of height $NL$ with the singular set
    \[F^{-1}(\mathcal{S}) \cap \partial \Omega \times [x, x + NL],\]
    is $\varepsilon$-dense when projected orthogonally onto $\partial \Omega$, for any $x$. This will be crucial to exclude the intersection of long maximal tubes in the angle estimate Lemma \ref{lem:angle_est} below. In other words, $NL$ quantifies the second part of Proposition \ref{prop:maximaltube} for periodic tubes.

    \begin{ex}\rm\label{ex:defN}
        Here we consider a few examples shedding some light on Definition \ref{def:N}. Firstly, since $SO(1)$ is trivial, we have $N = 1$ in the case $n = 1$ always. Then, if $\mathcal{R} \in SO(n)$ is rational, i.e. it is of finite order $o$, then the orbits $\mathcal{O}_i(r)$ consist of up to $o$ points and one easily sees that $N(\mathcal{R}, \varepsilon, r) = o$ for $\varepsilon$ small enough and $r \neq 0$. Finally, if we take $\mathcal{R} \in SO(2)$ to be a rotation in an irrational angle $\alpha$, the speed at which the orbit of a point fills the circle depends on the arithmetic properties of $\alpha$ -- see \cite[Chapter 4.3.]{KN74} for precise results. We also remark that we could have taken in \eqref{eq:defproperty} the $\mathcal{R}$ to be a reflection -- then for $\varepsilon$ small enough, we would have $N(\mathcal{R}, \varepsilon, r) = 2$ since the orbits are of size $1$ or $2$.
    \end{ex}

    Let now $D$ be the doubled polyhedron as before and let $U_\varepsilon$ be the $\varepsilon$-neighbourhood of $\mathcal{S}$. By Theorems \ref{th:general-dynamics} and \ref{th:non-periodic}, any trajectory not hitting $U_\varepsilon$ is contained in an immersed periodic tube with a certain convex cross-section. Consider any immersed tube $T$ of length $L$ with cross-section being a disc of radius $\frac{\varepsilon}{10}$. Recall that these tubes are given by immersing solid cylinders $F: B_{\frac{\varepsilon}{10}} \times \mathbb{R} \to D_0$. By definition, $F$ extends to an immersion $F: B_{\varepsilon} \times \mathbb{R}\to D_0$. Then we have the following lower bound on the angle of intersections of such periodic tubes in $D$.
    
    \begin{lemma}\label{lem:angle_est} Let $T_1$ and $T_2$ be two intersecting, periodic tubes of radii $\varepsilon/10$, lengths $L_1, L_2$ and associated rotations $\mathcal{R}_1, \mathcal{R}_2$, with geodesics $\gamma_1$ and $\gamma_2$ in their centre respectively, such that $\gamma_1$ and $\gamma_2$ do not intersect $U_\varepsilon$. More precisely, assume that a geodesic $\gamma_1'$ lying in $T_1$ parallel to $\gamma_1$ intersects a geodesic $\gamma_2'$ lying in $T_2$ parallel to $\gamma_2$ at a point $m \in D_0$, at an angle $\alpha$. Assume also $\gamma_1'$ is not parallel to $\gamma_2'$. Then the following bounds hold
    \begin{itemize}
            \item[1.] We have
            \begin{equation}\label{angleestimate}
    	        \frac{1}{\sin \alpha} \leq \frac{\min\big(N(\mathcal{R_1}, \varepsilon/5)L_1, N(\mathcal{R_2}, \varepsilon/5)L_2\big)}{\frac{4\varepsilon}{5}}.
    	    \end{equation}
            \item[2.] In particular, if $n = 2$ then 
            \begin{equation}\label{angleestimaten=2}
    	        \frac{1}{\sin \alpha} \leq \frac{\min(L_1, L_2)}{\frac{4\varepsilon}{5}}.
    	    \end{equation}
    	    \item[3.] In particular, if $P$ is rational and $o_1$ and $o_2$ are the finite orders of $\mathcal{R}_1$ and $\mathcal{R}_2$
                \begin{equation}\label{angleestimaterational}
    	            \frac{1}{\sin \alpha} \leq \frac{\min(o_1 L_1, o_2 L_2)}{\frac{4\varepsilon}{5}}.
    	        \end{equation}
    \end{itemize}
    \end{lemma}
        	\begin{figure}
    	    \centering
    	    	\begin{tikzpicture}[scale=0.8]
		\draw[thick] 
				(0,-4)--(0,4)
				node[below left]{$T_2'$}
				(7,-4)--(7,4)
				(0,0) to[out=-85,in=-95,distance=2cm] (7,0)
				(3.5, 4) node[below left]{$\gamma_2$}
				(0,-4) to[out=-85,in=-95,distance=1.9cm] (7,-4)
				(0,4) to[out=-85,in=-95,distance=1.9cm] (7,4)
				(0,4) to[out=85,in=95,distance=1.9cm] (7,4)
				;
		\draw[dashed,opacity=0.5] 
				(0,0) to[out=85,in=95,distance=2cm] (7,0)
				(0,-4) to[out=85,in=95,distance=1.9cm] (7,-4)
				;
		\draw[opacity=0.5]	
				(3.5, -4)--(3.5, 4) 
				;

		\fill[black!5, thick,opacity=0.5,draw=black]
				(1, -5)--(1, 3) --(5.5, 2.85) node[below left, black]{\small $S_2$} --
				(5.5, -5.15) -- (1,-5)
				   (5.5,-1.25) circle (1.5pt) 
				   (1,-1.05) circle (1.5pt) 
				;
		\draw[thick] (3.5, 0) circle (1pt) (3.5, 0.1) node[left]{$p$}; 
		\fill[black!10,draw=black]
				(1,-1.05) node[above left, black]{\small $p_1$}-- (5.5,-1.25) node[above right, black]{\small $p_2$} -- (3.5,0) -- (1,-1.05) 
				;
		\draw (3.25,-1.15) node[below, black]{\tiny $H$} circle (2pt) -- (3.5,0) 
			;
		\draw[thick,blue] 
				(0.3, 2.5) -- (1, 1.63)
				(5.5, -3.92) -- (6.49,-5.05) (6.3,-4.5) node{$\gamma_1$}
				(1, 1.63) circle(2pt) 
				(5.5, -3.92) circle(2pt) 
				;
		\draw[thick, blue, dashed]
				(1, 1.63)--(5.5, -3.92)
				;
		
    	\draw[blue] (5.5, -3.1) arc (90:130:0.82)
    	(5.3, -3.35) node[]{\small$\alpha$};
					
    	\draw[arrows=<->, thick](1,-4.29)--(5.5,-4.45);
    	\tikzstyle{ann} = [fill=white,font=\footnotesize,inner sep=1pt]
		\node[ann] at (3.25, -4.35) {$h_2$};
		
	\end{tikzpicture}
	    	\caption{Intersection of the tube $T_2'$ with the geodesic $\gamma_1$ at an angle $\alpha$. The strip $S_2$ of width $h_2$ in the direction of $T_2'$ and containing $\gamma_1$ is lightly shaded.}
    	    \label{fig:tubeintersection}
    	\end{figure}
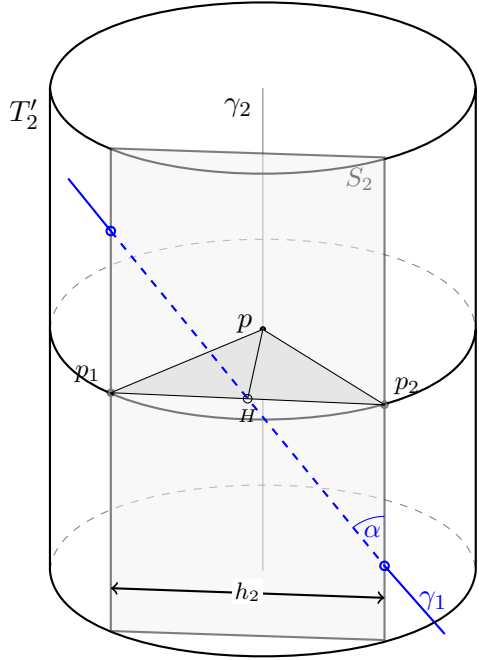
    \begin{proof}
        Enlarge $T_1$ to an immersed tube $T_1'$ of radius $\frac{3\varepsilon}{5}$, not intersecting $U_{\varepsilon/5}$, such that $T_1'$ intersects $\gamma_2$. This is possible since the distance between $\gamma_1$ and $\gamma_2$ is at most $\frac{\varepsilon}{5}$. Similarly, extend $T_2$ to an immersed tube $T_2'$ intersecting $\gamma_1$. By a slight abuse of notation, we will sometimes identify the tubes $T_j'$ with their corresponding unfoldings $(T_j')^\infty$ in $\mathbb{R}^n$ for $j = 1, 2$. Following this rule, let $p \in \gamma_2$ the unique, closest point to $\gamma_1$. It is unique since $\gamma_1$ and $\gamma_2$ are not parallel by the hypothesis.
        
        Consider the strip $S_2$ contained in $(T_2')^\infty$, determined by $\gamma_1$ and the direction of the tube $T_2'$ and denote its width by $h_2$. Consider the points $p_1$ and $p_2$ defined as intersections of the strip $S_2$ with the boundary of the cross-section of $(T_2')^\infty$ going through $p$ and consider the midpoint $H$ of the segment $p_1p_2$ (see Figure \ref{fig:tubeintersection}). By the triangle inequality applied to the triangle $p_1pH$, we have
    	\begin{equation}\label{eq:width}
    	    \frac{1}{2}h_2 > \frac{3\varepsilon}{5} - \frac{\varepsilon}{5} = \frac{2\varepsilon}{5}.
    	\end{equation}
    	
    	We introduce the parameter $l_1 := \frac{h_2}{\sin \alpha}$, i.e. the length of the portion of $\gamma_1$ lying in the strip $S_2$ (note that $\alpha \neq 0$). Introduce the notation $T_1''$ and $T_2''$ for the maximal periodic tubes corresponding to $T_1$ and $T_2$.
    	
    	We treat the case $n = 3$ first. Assume that $l_1 > N(\mathcal{R}_1, \varepsilon/5) L_1$ for the sake of contradiction. We may choose a segment $\gamma_1''$ of length $N(\mathcal{R}_1, \varepsilon/5) L_1$ parallel to $\gamma_1$, lying on the boundary of $(T_1'')^\infty$, such that $\gamma_1''$ is at a distance of (say) at least $\varepsilon/5$ to $\partial (T_2'')^\infty$. This can be done e.g. by looking at $S_2 \cap \partial (T_1'')^\infty$, i.e. translating $\gamma_1$ in the strip $S_2$ until we hit the boundary of $(T_1'')^\infty$. But by Lemma \ref{lemma:singsetdense} below (or the proof of Theorem \ref{th:general-dynamics}), there exists a singular point lying on the boundary of each slice of the periodic tube $(T_1'')^\infty$ of height $L_1$. By the invariance of the singular set under the map $(x, t) \mapsto (\mathcal{R}_1x, t + L_1)$ and by the definition of $N(\mathcal{R}_1, \varepsilon/5)$, such a singular point leaves a $\varepsilon/5$-dense trace when projected to the boundary of the cross-section $\Omega_1''$ of $T_1''$, the projections taken over the union of $N(\mathcal{R}_1, \varepsilon/5)$ copies of periodic slices in $(T_1'')^\infty$ of height $L_1$. Therefore, the $\varepsilon/5$-neighbourhood of the segment $\gamma_1''$ must contain singular points, which also lie in the interior of $(T_2'')^\infty$, giving a contradiction. 
    	
    	Therefore, we must have $l_1 \leq N(\mathcal{R}, \varepsilon/5) L_1$ and by an analogous argument, we obtain $l_2 := \frac{h_1}{\sin \alpha} \leq N(\mathcal{R}, \varepsilon/5) L_2$. We combine these two inequalities together with \eqref{eq:width} to get the estimate 1. for $n = 3$.
    	
    	For an arbitrary value of $n$, the proof is the same up to the point where we consider the intersection of the singular set with $\partial (T_1'')^\infty$. Denote by $\Omega_1''$ the cross-section of $T_1''$. Since $\Omega_1''$ is invariant under $\mathcal{R}_1$, we know that $\mathcal{R}_1$ acts on $\partial \Omega_1''$ and we may write
    	\[\partial \Omega_1'' = \cup_{i \in I} \mathcal{O}_i,\]
    	where $\mathcal{O}_i$ are minimal closed orbits of the action of $\mathcal{R}_1$ restricted to $\partial \Omega_1''$. By the proof of Lemma \ref{lemma:singsetdense}, if we denote by $\mathcal{T}_1$ the projection to $\partial \Omega_1''$ of the singular points in a slice of height $L_1$ in $\partial (T_1'')^\infty$, the projections of all singular points on $\partial (T_1'')^\infty$ are given by the forward orbit of $\mathcal{T}_1$ under $\mathcal{R}_1$. By the result of Lemma \ref{lemma:singsetdense}, this means that for each point $x \in \partial \Omega_1''$ and a $\delta > 0$, there exists a $y \in \mathcal{T}_1$ whose forward orbit under $\mathcal{R}_1$ is $\delta$-close to $x$. In other words, the union of orbits $\mathcal{O}_i$ containing singular points is dense in $\partial \Omega_1''$. Observe also that $\Omega_1'' \subset B_{\diam P}$, so the $(\varepsilon/5 + \delta)$-neighbourhood of the $\gamma_1''$ as constructed in the previous paragraph, for any $\delta > 0$, will contain singular points under the assumption that $l_1 > N(\mathcal{R}_1, \varepsilon/5) L_1$. But for a small enough $\delta > 0$, the $(\varepsilon/5 + \delta)$-neighbourhood of $\gamma_1''$ is contained in $T_2''$, giving a contradiction. This completes the missing step in the proof above, and proves the estimate 1. in full.

    	If $n = 2$, we have $N(\mathcal{R}, \varepsilon) = 1$ always since $SO(1)$ is trivial, so \eqref{angleestimaten=2} is a consequence of \eqref{angleestimate}. If $P$ is rational, for $\varepsilon > 0$ sufficiently small we have $N(\mathcal{R}, \varepsilon) = o$ where $o$ is the order of the element $\mathcal{R}$ (cf. Example \ref{ex:defN}), so \eqref{angleestimaterational} is a consequence of \eqref{angleestimate}.
    \end{proof}
    
    \begin{rem}\rm
        By the argument in the second to last paragraph in the proof of Lemma \ref{lem:angle_est}, we deduce that one may take $L(\varepsilon) = N(\mathcal{R}, \varepsilon) L$ in Lemma
        \ref{lemma:singsetdense}.
    \end{rem}
    
    We are in a position to prove the main asymptotic estimate on the lengths of closed orbits missing an $\varepsilon$-neighbourhood $U_\varepsilon$ of the singular set $\mathcal{S}$ in $D$. Note that by Theorem \ref{thm:fintubecond}, the number of  periodic tubes whose central geodesic misses $U_\varepsilon$ is finite -- denote this number by $M(P, \varepsilon) = M(\varepsilon) \in \mathbb{N}_0$. 
    \begin{theorem}\label{thm:partition}
    	Let $\varepsilon > 0$. Enumerate the immersed periodic tubes whose central geodesics do not hit $U_\varepsilon$ by $\{T_i : i = 1, \dotso, M(\varepsilon)\}$. Denote the length of $T_i$ by $L_i$ and the associated rotation by $\mathcal{R}_i$. 
    	Then, there exists a universal constant $C = C(n) > 0$, such that for $\varepsilon$ small enough
    	\begin{equation}\label{eq:cylindersbound0}
    	    \sum_{i = 1}^{M(\varepsilon)} \frac{1}{\big(N(\mathcal{R}_i, \varepsilon/5)\big)^{n-1} L_i^{n-2}} \leq C\frac{\vol(P)}{\varepsilon^{2n-2}}.
    	\end{equation}
    	In particular, if $P$ is rational or $\dim P = 2$ we have as $\varepsilon \to 0$
    	\begin{equation}\label{eq:cylindersbound}
    	    \sum_{i = 1}^{M(\varepsilon)} \frac{1}{L_i^{n - 2}} = O\big(\varepsilon^{-2(n - 1)}\big).
    	\end{equation}
    \end{theorem}
    \begin{proof}
    We first consider the case $n = 3$. The idea of the proof is to partition the phase space into small distinct volumes around periodic tubes and to estimate these suitably using Lemma \ref{lem:angle_est}. For this purpose, we introduce the sets
    \begin{equation}\label{eq:vi}
        V_i := \Big\{(x, \theta) \in SD_0 \,:\, x \in T_i, \, \,|\dot{\gamma}_i(x) - \theta| < \frac{4\varepsilon/5}{2 N(\mathcal{R}_i, \varepsilon/5) L_i}\Big\},
    \end{equation}
    where $\dot{\gamma}_i(x)$ is the unit speed of the geodesic at $x$ in the direction of the immersed tube $T_i$. The distance $|\cdot|$ is taken to be the spherical distance on the sphere $S_xD_0$ over $x$. We may take $T_i$ to have a disc cross-section of radius $\varepsilon/10$. Observe that the tubes are disjoint and have no self-intersections by Lemma \ref{lem:angle_est}. Then $V_i$ have volume equal to the one of a tube of radius $\varepsilon/10$ and length $L_i$ times a spherical cap of spherical radius $\frac{4\varepsilon/5}{2 N(\mathcal{R}_i, \varepsilon/5) L_i}$. Thus we may estimate the volume of $V_i$ in phase space
    \begin{align}\label{eq:volumes}
    \begin{split}
        \vol(V_i) &= L_i \times \Big(\frac{\varepsilon}{10}\Big)^2 \pi \times 2\pi \Big(1 - \cos\Big(\frac{4\varepsilon/5}{2 N(\mathcal{R}_i, \varepsilon/5) L_i}\Big)\Big)\\
        &= C' \varepsilon^2 \times L_i \times \sin^2\Big(\frac{\varepsilon}{5 N(\mathcal{R}_i, \varepsilon/5)  L_i}\Big)\\
        &\geq C'\frac{\varepsilon^4}{\big(N(\mathcal{R}_i, \varepsilon/5)\big)^2 L_i}.
    \end{split}
    \end{align}
    Here $C' > 0$ is a constant changing from line to line. 
    The last inequality in \eqref{eq:volumes} follows by observing that the lengths of closed orbits in $P$ are bounded from below by some $c(P) > 0$, so for $\varepsilon$ small enough we may apply $\sin t \geq t/2$ valid for $t \in [0, \frac{\pi}{3}]$. 
    
    Now taking the disjoint sets $V_i$ as in \eqref{eq:vi}, for $\varepsilon$ small enough and $i = 1, \dotso, M(\varepsilon)$, summing the volumes and using the estimate \eqref{eq:volumes}, we obtain the main asymptotic estimate \eqref{eq:cylindersbound0}. In the general case, the proof is completely analogous to the proof above and we will omit it. If $P$ is rational or $\dim P = 2$, the constants $N(\mathcal{R}, \varepsilon)$ are explicitly known (see Example \ref{ex:defN}), so one obtains \eqref{eq:cylindersbound}.
    \end{proof}
    
    \begin{rem}\rm
        In the case of polygons, i.e. $n = 2$, this lemma shows the cylinder (finite tube) condition of \cite{HHM} directly. More precisely, for the sum on the left of \eqref{eq:cylindersbound} to be finite, we would need it to contain finitely many elements, since for the case $n = 2$ every summand is exactly equal to one.
        Note that for $n = 3$, we do not immediately get the analogous claim since we have the $\frac{1}{(N(\mathcal{R}_i, c\varepsilon))^2 L_i}$ ``weights'' in \eqref{eq:cylindersbound0}. However, one may use a slightly different argument (see below).
    \end{rem}
    
    Using the angle estimate in Lemma \ref{lem:angle_est}, we may re-prove the finite tube condition obtained in Theorem \ref{thm:fintubecond}.
    
    \begin{proof}[Proof of Theorem \ref{thm:fintubecond} using the angle estimate]
    	To prove the claim, let us assume the contrary. That is, assume there are infinitely many periodic tubes $T_i$ of length $L_i$ with associated rotations $\mathcal{R}_i$, for $i \in \mathbb{N}$, that contain trajectories missing $U_\varepsilon$. Consider such trajectories $\gamma_i$ lying in $T_i$ and missing $U_\ep$, generated by some $(x_i, \theta_i) \in SD_0$. 
    	Since there are infinitely many distinct tubes, due to the compactness of the phase space we have that a subsequence, still called $(x_i, \theta_i)$ after relabelling, converges to $(x, \theta) \in SD_0$. By a straightforward continuity argument, the future trajectory $\gamma$ determined by $(x, \theta)$ also stays in $D_0 \setminus U_\ep$. By Theorems \ref{th:general-dynamics} and  \ref{th:non-periodic} such a geodesic $\gamma$ is contained in a periodic, limiting tube $T$ of length $L$ and with the associated rotation $\mathcal{R}$.
    	
    	Consider the immersed tubes $T_i', T'$ of radius $\varepsilon/10$ centred at $x_i, x$ and with directions given by $\theta_i, \theta$ (directions of $\gamma_i, \gamma$), respectively, for any $i \in \mathbb{N}$. 
    	If $\alpha_i = \angle (T_i', T') = \angle (\theta_i, \theta)$ is the angle of these intersecting tubes for each $i \in \mathbb{N}$, by applying Lemma \ref{angleestimate} we obtain
    	\begin{equation}\label{eq:lowerbound}
    	    \sin \alpha_i \geq \frac{4 \ep/5}{N(\mathcal{R}, \varepsilon/5) L}.
    	\end{equation}
    	Here $N(\mathcal{R}, \varepsilon/5)$ is the quantity defined in \eqref{eq:defproperty}. But since $(x_i, \theta_i) \to (x, \theta)$, we must have $\alpha_i \to 0$, which is a contradiction to \eqref{eq:lowerbound}.
    \end{proof}
    
    \begin{rem}\rm
        In the statements above, for simplicity we mostly ignored the situation where the associated map $\mathcal{R}$ to a periodic tube is in $O(n)\setminus SO(n)$ and instead dealt with the tube of doubled length whose associated map is orientation preserving. However, the analogous statement of Lemma \ref{angleestimate} holds for such tubes, and this may be applied to subsequent proofs of Lemma \ref{angleestimate} and Theorem \ref{thm:fintubecond} to obtain analogous results.
    \end{rem}
    
        
    \section{A control estimate with an almost periodic boundary condition}\label{sec:control_theory}
    
    In this section we prove the main control estimate, on a product space with an almost periodic boundary condition. Since we were unable to locate an appropriate reference, we start by establishing some control-theoretic preliminaries. As a general reference for the semiclassical analysis used in this section, the reader is referred to \cite{Zw} or \cite[Appendix E]{DZ}.
    

    To this end, let $(M_x, g_x)$ be a compact Riemannian manifold with Lipschitz boundary. Denote by $-\Delta_g$ the positive-definite Laplace-Beltrami operator. We say a subset $A \subset M_x$ satisfies the \emph{geometric control condition} or just \emph{(GCC)} if every geodesic $\gamma$ in $M_x$ hits $A$ in finite time.

    \begin{lemma}\label{lemma:auxiliary}
        Let $\omega \subset M_x$ be an open neighbourhood of the boundary, satisfying (GCC). There exists a $C = C(g_x, \omega) > 0$, such that for any $s \in \mathbb{R}$ and any $v$ satisfying 
        \begin{align}
            (-\Delta_{g_x} - s)v = g, \quad v|_{\partial M_x} = 0,
        \end{align}
        with $v \in H_0^1(M_x)$ and $g \in H^{-1}(M_x)$, we have the apriori estimate
        \begin{align}\label{eq:apriori}
            \lVert{v}\rVert_{L^2(M_x)} \leq C(\lVert{g}\rVert_{H^{-1}(M_x)} + \lVert{v|_{\omega}}\rVert_{L^2(\omega)}).   
        \end{align}
    \end{lemma}

    \begin{proof}
        Assume without loss of generality $g, v$ are real-valued. We split the proof to cases according to the value of $s \in \mathbb{R}$. We will use \emph{semiclassical defect measures} in the case of large $s$.
        
        \emph{Case 0: fixed $s$.} Assume inequality \eqref{eq:apriori} does not hold for this  fixed $s$, so there is a sequence $v_k \in H_0^1(M_x)$ with $\lVert{v_k}\rVert_{L^2(M_x)} = 1$ and $g_k \in H^{-1}(M_x)$ with
        \begin{align}
            \lVert{g_k}\rVert_{H^{-1}(M_x)} + \lVert{v_k|_\omega}\rVert_{L^2(\omega)} < \frac{1}{k}
        \end{align}
        for any $k \in \mathbb{N}$. Then clearly $\lVert{g_k}\rVert_{H^{-1}(M_x)} \to 0$ and $\lVert{v_k|_\omega}\rVert_{L^2(\omega)} \to 0$ as $k \to \infty$. Elliptic estimates give us that $\lVert{v_k}\rVert_{H^1(M_x)}$ is bounded uniformly in $k$,\footnote{By elliptic estimates, we have $\lVert{v}\rVert_{H^1(M_x)} \leq C'(\lVert{g}\rVert_{H^{-1}(M_x)} + \lVert{v}\rVert_{L^2(M_x)})$ for some $C' > 0$, for $v$ and $g$ as in the statement of the Lemma.} so by Rellich's theorem we may assume $v_k \to v$ in $L^2(M_x)$ after a possible re-labelling. By the assumptions, we have in the sense of distributions
        \[(-\Delta_{g_x} - s) v = 0, \quad v|_{\omega} = 0.\]
        Elliptic regularity gives $v$ is $C^\infty$ in the interior of $M_x$ and therefore by the unique continuation principle $v \equiv 0$, which contradicts $\lVert{v}\rVert_{L^2(M_x)} = 1$.
        
        \emph{Case 1: bounded $s$.} Here we show that the constant $C = C(s)$ in \eqref{eq:apriori} is locally bounded. The proof is by contradiction and is very similar to the previous case, so we omit it.
        
        \emph{Case 2: $s < -\varepsilon < 0$.} In this case $1 > \varepsilon > 0$ is fixed. Integrating by parts, we see
        \begin{align}\label{eq:intbyparts}
            \lVert{dv}\rVert^2_{L^2(M_x)} - s\lVert{v}\rVert^2_{L^2(M_x)} = \langle{g, v}\rangle_{H^{-1}(M_x) \times H_0^1(M_x)}.
        \end{align}
        Estimating right hand side using the boundedness of the $H^{-1} \times H_0^{1}$ pairing and using $s < - \varepsilon$, we obtain $\lVert{v}\rVert_{L^2(M_x)} \leq \frac{1}{\varepsilon} \lVert{g}\rVert_{H^{-1}(M_x)}$, proving \eqref{eq:apriori} for $s < -\varepsilon$.
        
        \emph{Case 3: $s \to \infty$.} Now it suffices to argue by contradiction and assume there is a sequence $s_k \to \infty$ with $C(s_k) \to \infty$, as $k \to \infty$. Then, there is a sequence $v_k \in H_0^1(M_x)$, with $\lVert{v_k}\rVert_{L^2(M_x)} = 1$ such that
        \begin{align}
            (-\Delta_{g_x} - s_k) v_k = g_k,
        \end{align}
        with $\lVert{g_k}\rVert_{H^{-1}(M_x)} \to 0$ and $\lVert{v_k}\rVert_{L^2(\omega)} \to 0$ as $k \to \infty$. We introduce a (small) semiclassical parameter $h_k > 0$ by $h_k^2 := \frac{1}{s_k}$. Then we have\footnote{Recall that in $\mathbb{R}^n$, $\lVert{u}\rVert_{H^{-1}_h} = \lVert{\langle{h\xi}\rangle^{-1} \widehat{u}(\xi)}\rVert_{L^2}$ and $\lVert{u}\rVert_{H^{-1}} = \lVert{\langle{\xi}\rangle^{-1} \widehat{u}(\xi)}\rVert_{L^2}$, where $\langle{\xi}\rangle = (1 + |\xi|^{2})^{\frac{1}{2}}$. From definitions and since $g_k = o_{H^{-1}} (1)$, we have locally $\lVert{h_k^2 g_k}\rVert_{H^{-1}_{h_k}} \leq h_k \lVert{g_k}\rVert_{L^2} = o(h_k)$ for $h_k < 1$ and as $k \to \infty$.}
        \begin{align}
            (-h_k^2 \Delta_{g_x} - 1) v_k = o_{H^{-1}_h}(h_k).
        \end{align}
        So by \cite[Theorems E.42, E.43, E.44]{DZ}, we know there is a semiclassical (Radon) measure $\mu$ on $S^*M_x$ associated to a subsequence $v_k$ (after re-labelling), such that it is invariant under the geodesic flow\footnote{In this proof we only work with the invariance of the semiclassical measure under the flow in the interior. For the invariance under the billiard (broken geodesic) flow, the reader is referred to \cite{ZZ96}.} of $g_x$ and
        \begin{align}\label{eq:semiclmeasure}
            \langle{a(x, h_kD) v_k, v_k}\rangle_{L^2(M_x)} \to \int_{S^*M_x} a d\mu, \quad \mathrm{\,\, as \,\,} k \to \infty.
        \end{align}
        Here $a \in C_0^\infty(T^*M_x)$, compactly supported in the interior of $M_x$ and $a(x, hD)$ is a $\Psi$DO obtained by the semiclassical quantisation procedure on (the interior of) $M_x$.
        
        Now, since $\lVert{v_k}\rVert_{L^2(\omega)} \to 0$ as $k \to \infty$, we have by \eqref{eq:semiclmeasure} that $\mu = 0$ on $\pi^{-1}(\omega)$, where\,\, $\pi: S^*M_x \to M_x$ is the projection. Since $\mu$ is invariant by the geodesic flow and $\omega$ satisfies (GCC), we thus have $\mu \equiv 0$. But since we assumed no mass escapes to the boundary, we have that $\mu$ is a probability measure on $S^*M_x$ (cf. Proposition \ref{prop:dis_supp}), which is a contradiction.
        \end{proof}
        
\begin{rem}\rm
    
    The case of Lemma \ref{lemma:auxiliary} for $M_x = [0, a]$ was considered in \cite{BZ}, by using elementary means to prove the inequality directly. Note also that if $M_x \subset \mathbb{R}^n$, then any open neighbourhood of the boundary automatically satisfies the (GCC).
\end{rem}

\begin{rem}\rm
    In the theorem above, the condition that $\omega$ is a neighbourhood of the boundary is used to prevent the mass to escape from the interior (see \cite{HT02} for a very similar condition, but in a different context). An example where Dirichlet eigenfunctions concentrate at the boundary is given by the upper hemisphere of $S^2 \subset \mathbb{R}^3$ and eigenfunctions
    \[u_l = c_l e^{i(l - 1)\varphi} (\sin \theta)^{l - 1} \cos \theta\]
    for $l \in \mathbb{N}$, of eigenvalue $\lambda = l(l + 1)$. Here $c_l \sim l^{3/4}$ is the $L^2$-normalisation constant and $(\theta, \varphi)$ are the spherical coordinates. Then one may show that $u_l$ concentrates on the equator $z = 0$. However, it seems that using the techniques as in \cite[Theorem 8]{BZ1}, one could deal with this and prove a more general version of Lemma \ref{lemma:auxiliary}, for $\omega$ any open set satisfying (GCC). 
\end{rem}

We are now in shape to prove the main theorem of this section, which generalises a result of N. Burq given in the Appendix \ref{app:B}. As before, let $(M_x, g_x)$ be a compact Riemannian manifold with Lipschitz boundary.

\begin{theorem}\label{thm:cont_per_cond}
    Let $\varphi: M_x \to M_x$ be an isometry. Assume that $u \in H^1_{\loc}(M_x \times \mathbb{R}) \cap C(\mathbb{R}, H^1(M_x))$, such that $u(x, t + L) = u(\varphi(x), t)$ for all $(x, t) \in M_x \times \mathbb{R}$, for some $L > 0$. Define $\mathcal{C}_\varphi := M_x \times [0, L] / (x, L) \sim (\varphi(x), 0)$ to be the mapping torus determined by $\varphi$, with the inherited Riemannian metric from $M_x \times \mathbb{R}$. Assume $u$ satisfies, for some $s\in \mathbb{R}$
    \begin{align}\label{eq:elleq1}
        (- \Delta_{g_x} - \partial_t^2 - s) u = f \,\,\, \mathrm{on} \,\,\, M_x \times \mathbb{R}, \quad u|_{\partial M_x \times \mathbb{R}} = 0,
    \end{align}
    where $f \in H_{\loc}^{-1}(M_x \times \RR) \cap C(\RR, H^{-1}(M_x))$. Let $\omega \subset M_x$ be an open neighbourhood of the boundary satisfying (GCC) and assume $\omega$ invariant under $\varphi$. Denote the mapping torus over $\omega$ by $\omega_\varphi$. Then there exists a constant $C = C(M_x, g_x, \omega) > 0$, such that the following observability estimate holds:
    \begin{align}\label{eq:mainineq}
        \lVert{u}\rVert_{L^2(\mathcal{C}_\varphi)} \leq C(\lVert{f}\rVert_{H^{-1}_xL^2_t(\mathcal{C}_\varphi)} + \lVert{u|_{\omega_\varphi}}\rVert_{L^2(\omega_\varphi)}).
    \end{align}
\end{theorem}
\begin{proof}
    We use the theory of almost periodic functions outlined in Section \ref{subsec:almostperiodic}. By Proposition \ref{prop:iso_adm} we know $\varphi$ is admissible, so by Lemma \ref{lemma:admissiblealmostperiodic} we have
    \begin{align*}
        u: \mathbb{R} \ni t \mapsto u(\cdot, t) \in H_0^1(M_x) \subset L^2(M_x)
    \end{align*}
    is almost periodic. We will use the shorthand notation $u(t) := u(\cdot, t) \in L^2(M_x)$ and $f(t) := f(\cdot, t)$. Therefore, there exists a countable set $\{\lambda_n\}_{n = 1}^\infty \subset \mathbb{R}$ such that
    \begin{align}\label{eq:apu}
        u(t) \sim \sum_{n = 1}^\infty a(\lambda_n; u) e^{i\lambda_n t}.
    \end{align}
    Recall the notation $\sim$ denotes a formal association of the series on the right hand side to the almost periodic function $u(t)$. The Fourier-Bohr transformation is given by
    \begin{equation}\label{eq:bohr}
        a(\lambda_n; u) = \lim_{T \to \infty}\frac{1}{2T} \int_{-T}^T u(t) e^{-i \lambda_n t} dt =: \mathcal{M} \{u(t) e^{-i\lambda_n t}\},
    \end{equation}
    where $\mathcal{M}\{\cdot\}$ denotes taking the mean value in $\mathbb{R}$. Therefore, we have $a(\lambda_n; u) =: u_n(x) \in H_0^1(M_x)$. 
    
    Now, since $\varphi$ an isometry so the pullback $\varphi^*$ commutes with $\Delta_{g_x}$, we have $f(x, t + L) = f(\varphi(x), t)$ for all $(x, t)$, so $f(t)$ is also almost periodic with range in $H^{-1}(M_x)$. One can check that, using mapping properties of $-\Delta_{g_x} - s: H_0^{1}(M_x) \to H^{-1}(M_x)$ and integration by parts for the $\partial_t^2$ factor, together with the convergence of the Bohr transform \eqref{eq:bohr}
    \begin{align*}
        (- \Delta_{g_x} - \partial_t^2 - s) u (t) \sim \sum_{n = 1}^\infty (-\Delta_{g_x}  - (s - \lambda_n^2)) u_n(x) \cdot e^{i\lambda_n t}.
    \end{align*}
    Similarly, one has $f(t) \sim \sum_n f_n(x) \cdot e^{i\lambda_n t}$ where $f_n(x)$ are the Fourier-Bohr coefficients of $f(t)$, and by uniqueness of Fourier-Bohr expansions of almost periodic functions, we have for all $n \in \mathbb{N}$
    \begin{align*}
        f_n \equiv (-\Delta_{g_x}  - (s - \lambda_n^2)) u_n.
    \end{align*}
    Intuitively, the above takes the place of a Fourier series expansion for the periodic case. Now for every $n \in \mathbb{N}$ and some $C > 0$, the estimate in Lemma \ref{lemma:auxiliary} yields
    \begin{align}\label{eq:aprioriap}
        \lVert{u_n}\rVert_{L^2(M_x)} \leq C(\lVert{f_n}\rVert_{H^{-1}(M_x)} + \lVert{u_n|_{\omega}}\rVert_{L^2(\omega)}).   
    \end{align}
    Now we sum the inequalities \eqref{eq:aprioriap} and use Parseval's identity for almost periodic functions \eqref{eq:parseval}
    \begin{align}
    \begin{split}
        \frac{1}{L}\lVert{u}\rVert_{L^2(\mathcal{C}_\varphi)}^2 &= \mathcal{M}\{\lVert{u(t)}\rVert^2_{L^2(M_x)}\} = \sum_{n = 1}^\infty \lVert{u_n}\rVert^2_{L^2(M_x)}\\
        &\leq C\sum_{n=1}^\infty \big(\lVert{f_n}\rVert_{H^{-1}(M_x)}^2 + \lVert{u_n|_\omega}\rVert_{L^2(\omega)}^2\big)\\
        &= C\big(\mathcal{M}\{\lVert{f(t)}\rVert_{H^{-1}(M_x)}^2\} + \mathcal{M}\{\lVert{u|_{\omega}}\rVert^2_{L^2(\omega)}\}\big)\\
        &= \frac{C}{L} \big(\lVert{f}\rVert_{H_x^{-1}L^2_t(\mathcal{C}_\varphi)}^2 + \lVert{u|_\omega}\rVert_{L^2(\omega_\varphi)}^2\big).
    \end{split}
    \end{align}
    More precisely, we applied Parseval's identity in the first line to $u$, in the third line to $f$ and $u|_{\omega \times \mathbb{R}}$. Here we also used the identity 
    \[u|_{\omega \times\mathbb{R}} \sim \sum_{n = 1}^\infty u_n|_\omega e^{i\lambda_nt},\]
    which follows upon restricting \eqref{eq:apu} to $\omega \times \mathbb{R}$ and recalling that $\varphi: \omega \to \omega$ is an isometry. We used \eqref{eq:aprioriap} in the second line. In the first and last lines we also used the fact that $\lVert{u(t)}\rVert^2_{H^s(M_x)}$ and $\lVert{f(t)}\rVert^2_{H^s(M_x)}$ are periodic with period $L$ for any allowed $s$ (as $\varphi^*$ is an isometric isomorphism on $H^s(M_x)$), so applying $\mathcal{M}\{\cdot\}$ gives their mean over the interval $[0, L]$. This finishes the proof.
\end{proof}
\begin{rem}\rm
    It might be possible to prove the statement above by using a more direct approach of approximating $u$ by periodic functions on $M_x \times [0, kL]$ for $k$ large, applying Theorem \ref{thm:controlperiodic} and then taking a limiting procedure.
\end{rem}

\begin{rem}\rm\label{rem:spectrum_C_phi}
    It is interesting to compute the spectrum of $\mathcal{C}_\varphi$ in terms of the spectrum of $M_x$ and the map $\varphi$. In fact, as the pullback $\varphi^*$ and $-\Delta_{g_x}$ commute, and since $\varphi^*$ is orthogonal as a map on $L^2$, we may restrict $\varphi^*$ to eigenspaces $E_{\mu_n}$ of $-\Delta_{g_x}$ and thus choose an orthonormal Dirichlet eigenbasis of $M_x$, call it $e_n$, such that $-\Delta_{g_x}e_n = \mu_n e_n$ and $\varphi^*e_n = e^{i\nu_n L} e_n$. The reals $\nu_n$ are defined modulo $\frac{2\pi \mathbb{Z}}{L}$. We set for $n \in \mathbb{N}$ and $k \in \mathbb{Z}$
    \[e_{n, k}(x, t) := e_n(x) e^{it(\nu_n + \frac{2k \pi}{L})},\]
    and observe that $e_{n, k}$ descend to $\mathcal{C}_\varphi$. Also 
    \[(-\Delta_{g_x} - \partial_t^2) e_{n, k} = \Big(\mu_n + (\nu_n + \frac{2k\pi}{L})^2\Big) e_{n, k},\]
    and one may show that $e_{n, k}$ form an orthogonal eigenbasis of $L^2(\mathcal{C}_\varphi)$, with $\lVert{e_{n, k}}\rVert_{L^2}^2 = L$. Observe that a function $u$ in \eqref{eq:apu}, since it descends to $\mathcal{C}_\varphi$, satisfies $u_n(\varphi(x)) = u_n(x) e^{i\lambda_n L}$. In fact, density of $e_n$ may be used to show $e^{i\lambda_nL} = e^{i\nu_{k_n} L}$ for some $k_n \in \mathbb{Z}$ and so $\lambda_n \in \nu_{k_n} + \frac{2\pi \mathbb{Z}}{L}$. This determines the spectrum of $\mathcal{C}_\varphi$ as advertised.
    
    Given $u$ and $f$ as in Theorem \ref{thm:cont_per_cond}, the argument above gives us a way to write $u = \sum_n u_n(x) e^{i\lambda_nt}$ and $f = \sum_n f_n(x) e^{i\lambda_nt}$. However, the $e^{i\lambda_nt}$ are in general not orthogonal with respect to $L^2(0, L)$, which would take us eventually to the operation $\mathcal{M}\{\cdot\}$. To prove \eqref{eq:aprioriap}, we would need a Parseval's identity (cf. \eqref{eq:parseval}). Therefore, in order to avoid the use of almost periodic theory, we would have to re-prove some essential parts of it.
\end{rem}


Having the above estimate with almost periodic boundary conditions at hand, we can recover the control theoretic result with periodic boundary conditions proved and used in \cite{HHM}, which stems from the work of Burq-Zworski \cite{BZ, BZ1}. More precisely, we have

\begin{corollary}
    If we choose $\varphi = \id$ in Theorem \ref{thm:cont_per_cond}, we recover \cite[Proposition 14]{HHM} and \cite[Proposition 6.1]{BZ1}. At the same time, we recover Theorem \ref{thm:product} for the case $M_y = [0, L]$.
\end{corollary}

    \section{Proof of the Main Theorem}\label{sec:mainthm}
    
    In this section, we give the proof of the Main Theorem \ref{thm:eigenf_conc}. The proof is based on a careful analysis of the closed orbits missing a neighbourhood of the singular set, semiclassical measures and the control estimate from the previous section.
    
    Throughout the rest of the paper, we will slightly abuse notation, and treat $U$ both as a neighbourhood of the singular set $\mathcal{S}$ in the double $D$ and in $P$. It will be clear from the proofs that the main eigenfunction concentration result will also be established for the double $D$, but it will readily imply Theorem \ref{thm:eigenf_conc}. In fact, an eigenfunction $u_n$ on $P$ induces a $C^\infty$ eigenfunction on $D_0$ (denoted by the same letter), that is equal to $u_n$ on $P$ and to $-u_n \circ \sigma$ in $\sigma P$ for the Dirichlet Laplacian and $u_n \circ \sigma$ for the Neumann Laplacian. A lower bound on the mass of $u_n$ near $\mathcal{S}$ in $D$ translates to a similar bound in $P$.
    
    Also, to clarify, one can define the Laplacian on $D_0 := D \setminus \mathcal{S}$, by taking the Friedrichs extension of the operator with domain $C_c^\infty(D_0)$. It is known that this is self-adjoint with compact resolvent, so has discrete spectrum $\lambda_j$ going to infinity, and a complete orthonormal $L^2$-basis $u_j$ of Laplace eigenfunctions.  
    
    To begin the proof of our Main Theorem \ref{thm:eigenf_conc}, let us assume to the contrary that there is no concentration in the neighbourhood $U$, that is, there exists a subsequence $u_n$ satisfying
    \begin{equation}\label{eq:assumption}
    \lim_{n \rightarrow \infty} \int_U |u_n|^2 = 0.
    \end{equation}

    Let $\mu$ be an arbitrary semiclassical measure associated to the sequence $u_n$ in the standard way. That is, for any $a \in C^\infty_0(S^*D_0)$ (unit cotangent bundle on $D_0 = D \setminus \mathcal{S}$) we have
    \begin{equation} \label{eq:semicl_defect_meas}
        \lim_{n \to \infty} \langle a(x, h_nD) u_n, u_n \rangle = \int_{S^*D_0} a d\mu.
    \end{equation}{}
    Here $h_n^2 = \frac{1}{\lambda_n}$ and $a(x, hD)$ is a choice of quantisation on $D_0$. We will look at the interaction of this semiclassical measure with the geodesic flow on $D$.  To start with, we have the following result. Let $U_0 := U \setminus \mathcal{S}$ and denote by $\pi: S^*D_0 \to D_0$ the footpoint projection.
    \begin{prop}\label{prop:dis_supp}
    	The support of $\mu$ is disjoint from $\pi^{-1}(U_0)$ and $\mu$ is a probability measure which is invariant under the geodesic flow.
    \end{prop}
    \begin{proof}
    	Suppose to the contrary that there is a point $q \in \supp \mu$ with $\pi (q) \in U_0$. Hence by definition, there exists a neighbourhood $V$ with $q \in V \subseteq \pi^{-1}(U_0)$ so that $\mu(V) > 0$. Now let $\phi \in C_0^\infty(S^*D_0)$ be a function taking values in $[0, 1]$ with $\phi = 1$ on $V$ and $\supp \phi \subseteq \pi^{-1}(U_0)$. From the nonnegativity of $\mu$ and the assumption on $V$, it follows that $\langle \mu, \phi \rangle > 0$. Hence, using the definition of a semiclassical measure, we obtain the existence of a fixed positive constant $\gamma$ such that
    	\begin{equation} \label{eq:pos_semcl_meas}
    	    \lim_{n \to \infty} \langle \phi(x, h_nD) u_n, u_n \rangle_{L^2(D)} > \gamma > 0.
    	\end{equation}{}
    	
    	Now, we can decompose
    	\begin{align} \label{eq:decomp_semiclas_meas}
    	    \langle \phi(x, h_nD) u_n, u_n \rangle_{L^2(D)} = \langle \phi(x, h_nD) u_n, u_n \rangle_{L^2(\pi(V))} + \langle \phi(x, h_nD) u_n, u_n \rangle_{L^2(D \setminus \pi(V))}.
    	\end{align}
    	To estimate the right hand side further from above one can use a bound on the operator norm (see, e.g., \cite[Theorem 5.1]{Zw}):
    	\begin{equation}
    	    \| \phi(x, h_nD) \|_{L^2 \rightarrow L^2} \leq C \sup_{S^*D_0} |\phi| + O(h^{\frac{1}{2}}) = C + O(h^{\frac{1}{2}}).
    	\end{equation}{}
    	By the pseudolocality of PDOs (see, e.g., the Remark on p. 81, \cite{Zw}) and eventually slightly enlarging $V$, one can assume that the second term on the right in (\ref{eq:decomp_semiclas_meas}) can be made negligible, whereas the first term on the right in (\ref{eq:decomp_semiclas_meas}) is bounded above as:
        \begin{equation}
             \langle \phi(x, h_nD) u_n, u_n \rangle_{L^2(\pi(V))} \leq \left( C + O(h^{\frac{1}{2}}) \right) \| u_n \|^2_{L^2(\pi(V))} \rightarrow 0, \quad n \rightarrow \infty,
        \end{equation}{}
    	where we have used the assumption (\ref{eq:assumption}). This contradicts (\ref{eq:pos_semcl_meas}).

    	Further, the invariance of $\mu$ under the geodesic flow is a well-known fact for semiclassical defect measures. For further discussion on semiclassical correspondence and Egorov's theorem, we refer to \cite[Chapter 11]{Zw}.
    	
        To prove that $\mu$ is a probability measure, one follows a similar line of reasoning as above, i.e. by taking an appropriate cut-off functions supported in $D \setminus U_0$ and using (\ref{eq:semicl_defect_meas}) and the fact that $u_n$ are $L^2$-normalised. Technically, one needs to argue in a similar fashion as above (using pseudolocality) to estimate the different terms. We emphasise that it is crucial for the arguments that we work away from the singular points.
    \end{proof}

  
      
   
    


    Now we begin with the proof of Theorem \ref{thm:eigenf_conc}. Before tackling the general case, we first give an easier proof that works for rational polyhedra, and contains some preliminary ideas. This proof is essentially a higher dimensional analogue of \cite{HHM}.

    \begin{proof}[Proof of Theorem \ref{thm:eigenf_conc} for Rational Polyhedra]
    	Let $\mu$ be a limit semiclassical measure associated to the eigenfunction sequence $u_n$ as above, and let $(x_0, \zeta) \in {S^*D_0}$ be in the support of $\mu$. According to Proposition \ref{prop:dis_supp} and Corollary \ref{lem:Dichotomy}, the finite tube condition (see Theorem \ref{thm:fintubecond}) implies that $x_0$ belongs to a tube periodic in the direction of $\zeta$. Thus, the support of $\mu$ is included in the union of the lifted maximal  periodic tubes $T_i$, for $i = 1, 2, \dotso, N$. Furthermore, assume $U$ is the $\varepsilon$-neighbourhood $U_\varepsilon$ of $\mathcal{S}$ for an $\varepsilon > 0$.
    	
    	Choose one such tube $T := T_i$ and choose its length $L$ such that every geodesic in $T$ is periodic with period $L$. By definition, $T$ is the image under a local isometry $F: \Omega \times \mathbb{R} \to D_0$, where $\Omega$ is a convex polygon. Denote the coordinates on $\Omega$ by $(x, y)$ and on $\mathbb{R}$ by $z$. Using this local isometry $F$, we pull back the eigenfunction $u_n$ to $T$. We now apply the black box concentration techniques of \cite{BZ, BZ1, HHM} to this setting. We choose an ``inner shell'' of thickness $\varepsilon$ (that is, the complement of a tubular neighbourhood of $\pa \Omega$ of thickness $\varepsilon$), denoted by $\Omega_\varepsilon \subset \Omega$, so that $\Omega_\varepsilon \times \mathbb{R}$ does not intersect $F^{-1}(U_\varepsilon)$. Now consider a smooth cut-off function $\chi$ such that $\chi \equiv 1$ inside $\Omega_\varepsilon$ and $\chi = 0$ outside $\Omega_{\varepsilon/2}$. Then, $v_n := \chi u_n$ vanishes near the walls of the tube $T$. The $v_n$ are bounded in $L^2_{\loc}$, so there exists at least one associated semiclassical measure $\nu$ on $T$. Since $\mu$ is supported on a finite number of lifted tubes, there are only a finite number of directions in the support of $\nu$. So we can find a (constant-coefficient) semiclassical pseudodifferential operator $\Phi = a(hD)$ on $T$ that is microlocally $1$ in a neighbourhood of $\zeta = dz$ (the direction of $T$), but vanishes microlocally in a neighbourhood of every other direction in the support of $\nu$. Let $\Phi_n := a(h_n D)$.
    	
    	Now, consider the sequence of functions $w_n := \Phi_n v_n$ on $T$. The semiclassical measure $\nu'$ associated to this sequence is related to the one for the sequence $v_n$ by $\nu' = |\sigma (\Phi)|^2 \nu$, where $a = \sigma (\Phi)$ is the principal symbol of the operator $\Phi$. Thus, the support of $\nu'$ is restricted to directions parallel to $d z$ and to the geodesics parametrised by $x, y$ such that $\chi(x, y) = 1$.
    	
    	Then, since $\Phi$ commutes with constant coefficient differential operators, we have on $\Omega \times [0, L]$
    	\beq \label{eq:before_reduction}
    	\begin{split}
    	(-\Delta_\Omega - \pa_z^2 - \lambda_n) w_n &= - \Phi_n ((\Delta \chi) u_n) - 2\Phi_n (\pa_x \chi\pa_x u_n) -2\Phi_n (\pa_y \chi\pa_y u_n),\\
    	w_n|_{\Omega \times \{0\}} &= w_n|_{\Omega \times \{L\}.}
    	\end{split}
    	\eeq
    	
    	We are now in a position to apply Theorem \ref{thm:controlperiodic} (or Theorem \ref{thm:cont_per_cond} for $\varphi = \id$), for $u = w_n$, $s = \lambda_n$, $f = f_n: = - \Phi_n ((\Delta \chi) u_n) - 2\Phi_n (\pa_x \chi\pa_x u_n) -2\Phi_n (\pa_y \chi\pa_y u_n)$ and $\omega$ contained in the set $\{ \chi = 0\}$. By our choice of $\omega$, we have $\| w_n\|^2_{L^2(\omega \times [0, L])} \to 0$. Also, it is clear from the choice of our cut-off that since $\supp (\nabla \chi)$ is disjoint from $\supp \nu'$, $\| f_n\|_{H^{-1}_{x, y}L^2_z(\Omega \times [0, L])} \to 0$ as $n \to \infty$. This implies that $\nu'$, and hence $\nu$ does not have any mass in the direction of the tube $T$. 
    	
    	We may apply the above argument to each tube $T_i$ to obtain $\mu = 0$, which contradicts the latter conclusion of Proposition \ref{prop:dis_supp} and finally proves the claim.
    \end{proof}	 
    
    Now we give the proof in the general case, that is, the case of irrational polyhedra. Observe that in this setting there are several issues that render the above proof invalid. Firstly, the boundary condition, as used in Theorem \ref{thm:controlperiodic}, is not periodic, but almost periodic in the irrational case.   We recall and emphasise that in a periodic tube, trajectories need not be individually periodic. Moreover, we improve the proof by relying on the Property \eqref{eq:P2'}, that is significantly weaker than the finite tube condition.

   \begin{proof}[Proof of Theorem \ref{thm:eigenf_conc}]
        As before, let $\mu$ be the semiclassical measure associated to the eigenfunction sequence $u_n$, and let $(x_0, \zeta) \in S^*D_0$ be in the suppport of $\mu$. According to Proposition \ref{prop:dis_supp} and Corollary \ref{lem:Dichotomy}, this generates a trajectory contained in a maximal periodic tube $T$ with the direction of $\zeta$, of length $L$ and associated rotation $\mathcal{R}$. By definition, there is a local isometry $F: \Omega \times \mathbb{R} \to T$, where $\Omega$ is a convex set, invariant under $\mathcal{R}$. Using this local isometry, by a slight abuse of notation we identify $\Omega \times \mathbb{R}$ with $T$ and $F^*u_n$ with $u_n$; we also pull back the measure $\mu$ to $T$. Write $x$ for the coordinate on $\Omega$ and $t$ for the $\mathbb{R}$ coordinate. Note that $F \circ \varphi = F$, where $\varphi(x, t) = (\mathcal{R}x, t + L)$; thus $\mu$ is invariant under $\varphi$. Recall $\Omega_\varepsilon \subset \Omega$ denotes the inner shell of thickness $\varepsilon$ inside $\Omega$, i.e. the complement of the $\varepsilon$-neighbourhood of $\partial \Omega$.
        
        Let $\chi \in C_0^\infty(\Omega)$ be invariant by $\mathcal{R}$ and such that
        \begin{align}\label{eq:chi}
        \chi(x) = 
            \begin{cases}
                1,& x \in \Omega_{\varepsilon/2 + \eta}.\\
                0,& x \in \Omega\setminus \Omega_{\varepsilon/2}.
            \end{cases}
        \end{align}
        Denote $v_n := \chi u_n$ and by $\mathcal{C}_{\varphi}$ the mapping torus determined by $\varphi$. Since $T$ does finitely many reflections in its one period, we have $\lVert{v_n}\rVert_{L^2(\mathcal{C}_{\varphi})}$ bounded uniformly in $n$, and the associated semiclassical measure is $\nu = \chi^2 \mu$.
        
        Take now $\Phi$ a constant coefficient semiclassical PDO in $\mathbb{R}^n$, microlocally cutting off near $dt$, constructed as follows. Using the Property \eqref{eq:P2'}, we may take its symbol $a(\xi)$ supported in a small cone $\Gamma \subset \mathbb{R}^n\setminus 0$ around $dt$, such that all lines in the direction of $\Gamma$ with basepoint $x \in \Omega_{\varepsilon/2 - \eta} \setminus \Omega_{\varepsilon/2 + \eta}$ hit the set $F^{-1}(U_\varepsilon)$ in finite time (see Figure \ref{fig:tubemaintheorem}). Moreover, we may take $a(\xi)$ equal to $1$ near $dt$ and invariant under rotations around $dt$, i.e. so that we have $a\circ\mathcal{R} = a$.
         
         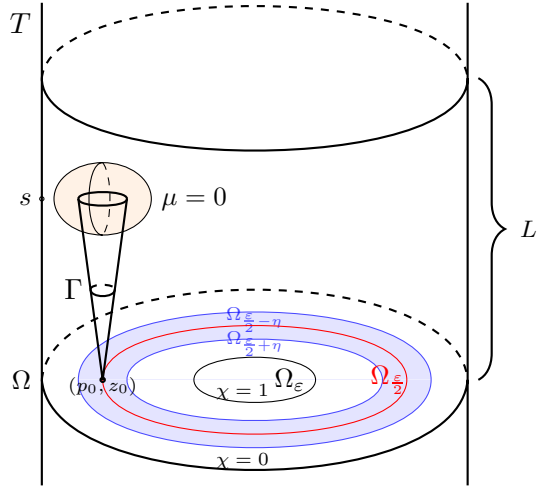
\begin{figure}
             \centering
             	\begin{tikzpicture}[scale = 0.8]
		\draw[thick]
				(0,-1.75)--(0,6.25)
				node[below left]{$T$}
				(7,-1.75)--(7,6.25)
				(0,0) node[left]{\small $\Omega$} to[out=-85,in=-95,distance=2cm] (7,0)
				(0,5) to[out=-90,in=-90,distance=1.6cm] (7,5)
				;
		\draw[thick,dashed]
				(0,0) to[out=85,in=95,distance=2cm] (7,0)
				(0,5) to[out=85,in=95,distance=1.6cm] (7,5)
				;
					
		\fill[blue!15, even odd rule, opacity = 0.7, draw=blue]
					(1-0.4,0) to[out=-90,in=-90,distance=1.5cm] (6.4,0) (3.5, 1) node[blue]{\tiny$\Omega_{\frac{\varepsilon}{2}-\eta}$}
					(1-0.4,0) to[out=90,in=90,distance=1.5cm] (6.4,0)
					
					(1+0.4,0) to[out=-90,in=-90,distance=.9cm] (6-.4,0) (3.5, 0.6)node[blue]{\tiny$\Omega_{\frac{\varepsilon}{2}+\eta}$}
					(1+0.4,0) to[out=90,in=90,distance=.9cm] (6-.4,0)
					;

		\draw[red]
					(1,0) to[out=-90,in=-90,distance=1.2cm] (6,0) (5.7, 0) node[]{\small$\Omega_{\frac{\varepsilon}{2}}$}
					(1,0) to[out=90,in=90,distance=1.2cm] (6,0);
					
		\draw
					(2.5,0) to[out=-90,in=-90,distance=.5cm] (4.5,0) node[left]{\small$\Omega_{\varepsilon}$}
					(2.5,0) to[out=90,in=90,distance=.5cm] (4.5,0)
					(3.3, -0.2) node[]{\tiny $\chi = 1$}
					(3.3, -1.35) node[]{\tiny $\chi = 0$}		
					;	
			
			\draw [thick, decorate,decoration={brace,amplitude=10pt,mirror,raise=4pt}, yshift=0pt]
(7,0) -- (7,5) node [black,midway,xshift=0.8cm] {\footnotesize
$L$};

		\fill[orange!10!white, draw=black]
				(1,3) ellipse (0.8cm and 0.6cm)
				(1,3.6) to[out=-180,in=180,distance=0.3cm] (1,2.4)
				(1.8, 3) node[right, black]{\small{$\mu = 0$}}
				;
		\draw (0, 3) circle (1pt) node[left]{\small$s$};
	
		\draw[thick] (1-.4,3) -- (1,0) (1.4,3)--(1,0) circle (1pt) (1, -.1)node[]{ \tiny\textbf{$(p_0,z_0)$}};
		\draw[thick] 
		(1-.4,3) to[out=-90, in=-90, distance= 0.15cm] (1+.4, 3)
		(1-.4,3) to[out=90, in=90, distance= 0.15cm] (1+.4, 3)
		(1-.2,1.5) to[out=-90, in=-90, distance= 0.15cm] (1+.2, 1.5)
		(1-.15,1.5) node[left]{$\Gamma$}
		;
		\draw[thick, dashed]
		(1-.2,1.5) to[out=90, in=90, distance= 0.1cm] (1+.2, 1.5)
		;
		\draw[black,dashed]
				(1,3.6) to[out=-50,in=50,distance=0.25cm] (1,2.4);
				;
	        \end{tikzpicture}
             \caption{The periodic tube $T$ of length $L$ with disc cross-section $\Omega$, point $(p_0, z_0) \in \partial \Omega_{\frac{\varepsilon}{2}}$, corresponding regions $\Omega_{\frac{\varepsilon}{2} \pm \eta}$, singular point $s \in \mathcal{S}$ and cone $\Gamma$. In orange is the set where $\mu = 0$.}
             \label{fig:tubemaintheorem}
         \end{figure}
        
        Let $w_n := \Phi_n(v_n)$, where $\Phi_n = a(h_nD)$ and observe that by Proposition \ref{prop:PHI}, $w_n$ descends to $\mathcal{C}_\varphi$. The semiclassical measure $\nu'$ associated to $w_n$ satisfies $\nu' = |a|^2 \chi^2 \mu$. By our choices of $\Phi$, $\chi$ and the invariance of $\mu$ under the geodesic flow, we have $\nu' = 0$ on $\Big(\Omega \setminus \Omega_{\varepsilon/2 + \eta}\Big) \times \mathbb{R}$. By a computation and since $\Phi$ commutes with constant coefficient differential operators, we have on $\Omega \times [0, L]$ the following equation with an almost periodic boundary condition
        
        \beq \label{eq:before_reduction_2}
        \begin{split}
    	    (-\Delta_\Omega - \pa_t^2 - \lambda_n) w_n &= - \Phi_n ((\Delta \chi) u_n) - 2\Phi_n (\nabla_x \chi \cdot \nabla_x u_n),\\
    	    w_n|_{\Omega \times \{0\}} &= (\mathcal{R}^*w_n)|_{\Omega \times \{L\}}.
        \end{split}
    	\eeq
    	\vspace{7pt}
    	
    	Choose $\omega \subset \Omega$ to be a small enough neighbourhood of $\partial \Omega$, for example $\omega = \Omega \setminus \Omega_{\varepsilon/2 - \eta}$ would work. We are now in a position to apply Theorem \ref{thm:cont_per_cond} to $u = w_n$, $s = \lambda_n$, $f = f_n := - \Phi_n ((\Delta \chi) u_n) - 2\Phi_n (\nabla_x \chi\cdot \nabla_x u_n)$ and $\omega$ as above to get\footnote{Note that a convex domain always has Lipschitz boundary, so we may apply the analytical results of previous section.}
    	\begin{align}\label{eq:controlineq}
    	    \|w_n\|_{L^2(\mathcal{C}_{\varphi})} \lesssim \|w_n\|_{L^2(\omega_{\varphi})} + \| \Phi_n ((\Delta \chi) u_n) + 2\Phi_n(\nabla_x \chi \cdot \nabla_x u_n)\|_{H^{-1}_x L^2_t(\mathcal{C}_{\varphi})}.
    	\end{align}
    	
    	By our choice of $\omega$, i.e. since $\chi = 0$ on a neighbourhood of $\overline{\omega}$ and so $\nu' = 0$ near $\overline{\omega}$, we have $\| w_n\|_{L^2(\omega_\alpha)} \to 0$. The $H^{-1}_{x}$ norm of the term for a fixed $t$ in the second bracket on the right hand side of \eqref{eq:controlineq} is bounded by a constant times the following expression, since $\Phi_n$ commutes with constant coefficient differential operators by Proposition \ref{prop:PHI}
    	\begin{align*}
    	    \|\Phi_n((\Delta \chi) u_n)\|_{H^{-1}(\Omega)} + \|\nabla_{x} \cdot (\Phi_n (u_n \nabla_{x} \chi) )\|_{H^{-1}(\Omega)}.
    	\end{align*}
    	This is further bounded by, for some $C > 0$
    	\begin{align}\label{eq:rhsphi}
    	    C\big(\|\Phi_n((\Delta \chi) u_n)\|_{L^2(\Omega)} + \|\Phi_n (u_n \nabla_{x} \chi) \|_{L^2(\Omega)}\big).
    	\end{align}
    	The $L^2_t$ norm of the expression in \eqref{eq:rhsphi} over $[0, L]$ goes to zero as $n \to \infty$, since the semiclassical measure $\nu'$ vanishes on $\Big(\Omega_{\varepsilon/2} \setminus \Omega_{\varepsilon/2 + \eta}\Big) \times \mathbb{R}$ and $\supp \nabla \chi \subset \Omega_{\varepsilon/2} \setminus \Omega_{\varepsilon/2 + \eta}$. This shows that the right hand side of \eqref{eq:controlineq} goes to zero as $n \to \infty$, which implies $\nu' \equiv 0$, and further implies $\mu \equiv 0$ since the choice of $(x_0, \zeta)$ was arbitrary. But this is a contradiction with Proposition \ref{prop:dis_supp}, which gives that $\mu$ is a probability measure.
    \end{proof}
   
   \subsection{Another application} In proving Theorem \ref{thm:eigenf_conc}, our main technical result was a version of a control theory estimate in \cite{BZ1}, which also works for periodic boundary conditions, as observed in \cite{HHM}. We have proved an analogous version in arbitrary dimensions, but more importantly, for almost periodic boundary conditions (see Theorem \ref{thm:cont_per_cond}). Among other consequences of such a control result is a stronger and more general version of \cite[Theorem 2]{M}, which has been stated in the particular context of partially rectangular billiards.
   \begin{theorem}\label{thm:conc_part_rect} 
        For any  periodic tube $T$ immersed in a convex $\Omega \subset \mathbb{R}^n$, and a neighbourhood $U$ of $\pa T$ inside $\Omega$, there exists a constant $C > 0$ such that 
   \beq\label{eq:conc_part_rect}
   -\Delta u = \lambda u  \Longrightarrow \int_U |u(x)|^2 dx \geq C \int_T |u(x)|^2dx,
   \eeq
   that is, no eigenfunction can concentrate in $T$ and away from $\partial T$.  
   \end{theorem}
   
    Here by immersing the  periodic tube $T$ into $\Omega$, we assume that the parts of the boundary $\partial \Omega$ where $T$ hits reflects are flat, i.e. consist of a hypersurface. The proof of Theorem \ref{thm:conc_part_rect} is absolutely similar in spirit to the proof in \cite{M}, except that in higher dimensions, we need to use our version of the control result, given by Theorem \ref{thm:cont_per_cond}.  We also note that the $L^2(T)$ norm is equivalent to the norm of the pullback over the mapping cylinder corresponding to $T$. We will skip the details.
   
   \section{Acknowledgements} The work was started when the third author was at the Technion, Israel supported by the Israeli Higher Council and the research of the third author leading to these results is part of a project that has received funding from the European Research Council (ERC) under the European Union's Horizon 2020 research and innovation programme (grant agreement No 637851). The work was continued when the third author was at the Courant Institute of Mathematical Sciences funded by Fanghua Lin and the Max Planck Institute for Mathematics, Bonn, and finished at IIT Bombay. He wishes to deeply thank them all. The first two authors gratefully acknowledge the Max Planck Institute for Mathematics, Bonn for providing ideal working conditions and for financial support. During the last stages of this project the first author was in Orsay, where he received funding from the European Research Council (ERC) under the European Union’s Horizon 2020 research and innovation programme (grant agreement No. 725967). The second author would like to thank the Fraunhofer Institute, IAIS, for the perfect working conditions and environment. The authors also thank Steve Zelditch and Luc Hillairet for very helpful conversations. Lastly, the authors are deeply grateful to both the Referees for significant improvements in the presentation. In particular, the idea of Proposition \ref{prop:iso_adm}, was suggested by one of the Referees, whose comments also led us to add Remark \ref{rem:spectrum_C_phi} and a new proof of Theorem \ref{th:non-periodic} in Appendix A.

    \appendix
    
\section{Some dynamical statements}\label{app:A}
    In this appendix for reader's convenience we prove the statements from Section \ref{subsec:dynamicsrevision} and give original references. We emphasise that our proof of Theorem \ref{th:non-periodic} contains our own original ideas.

    We start by recording the fact that any trajectory {\em in a polygon} whose forward closure does not contain any singular point must be periodic. In other words, no non-periodic trajectory can continue indefinitely inside the polygon without coming arbitrarily close to the pockets. The proof appears in \cite[Theorem 2 and Corollary 2]{GKT} for polygons; the statement follows from the more general Theorem \ref{th:non-periodic}.
	
	\begin{theorem}\label{thm:period_traj'}
	For an arbitrary polygon $P$, and for any $x \in T\Gamma_1$, if the closure of the set $ \pi \{f^i x : i \geq 0 \}$ does not contain any vertex of the polygon, then the orbit of $x$ is periodic, where $\pi : TP \to P$ is the usual projection map. 
	\end{theorem}
	
	We proceed with the proof of Theorem \ref{th:general-dynamics}. Here we adapt the proof of \cite[Theorem 5]{GKT}, written down in \cite{GKT} for $n = 3$, to higher dimensions. One distinction to the $n = 3$ and $n \geq 4$ cases in the theorem is that we cannot simply classify all the possible cross-sections of   periodic tubes (cf. Remark \ref{rem:general-dynamics}).
    \begin{proof}[Proof of Theorem \ref{th:general-dynamics}]
        Let us enumerate the faces of $P$ by $\mathcal{F}_1, \dotso, \mathcal{F}_l$. Given an unfolding $P = P^0, P^1, P^2, \dotso$, we write $R_i^{(j)}$ for the (affine) reflection in the $i$-th face in $P^j$ and denote the corresponding (linear) reflection through the origin by $R^{(j)}_{i, 0}$. By Proposition \ref{prop:maximaltube}, we have that $X(w)$ is convex and consists of parallel vectors $v \in \mathbb{R}^n$ with basepoints on one particular face; by Lemma \ref{lem:symbol_same_direction_same}, we have $f^k: X(w) \to X(w)$. Here $f$ is the first return map to a face. We assume first $k$ is even. Let $w = (i_0, \dotso, i_{k-1}, i_0, \dotso)$ where $(i_0, \dotso, i_{k-1})$ repeats itself and define the rotation
        \[\mathcal{R}_0 = R_{i_{k-1}, 0}^{(k-1)} \cdots R_{i_0, 0}^{(0)}.\]
        Since $f^k : X(w) \to X(w)$, we must have $\mathcal{R}_0v = v$. We will also write
        \[\mathcal{R} = R_{i_{k-1}}^{(k-1)} \cdots R_{i_0}^{(0)} = \tau + \mathcal{R}_0.\]
        Here $\tau \in \mathbb{R}^n$ is a translation vector. 
        
        We may extend $f^k$ by continuity to $f^k: \overline{X(w)} \to \overline{X(w)}$. By Brouwer's fixed point theorem, we must have at least one fixed point $x(w)$, i.e. $f^k(x(w)) = x(w)$. We may write $x(w) = (x_0, v)$ where $x_0 \in \mathcal{F}_{i_0}$. In fact, by unfolding the trajectories we see, for some $L > 0$
        \[\mathcal{R}x_0 = x_0 + Lv = \tau + \mathcal{R}_0 x_0.\]
        We may choose $x_0$ as the origin, to obtain $\tau = Lv$. The argument above also shows $f^k(X(w)) = X(w)$, since the image of $f^k$ is obtained by a rigid motion.
        
        Let the projection to the plane through origin normal to $v$ be denoted by $\pi$. Let $\mathcal{O} \subset \mathcal{F}_{i_0}$ denote the basepoints of $X(w)$. The cross-section $\Omega$ is obtained by projecting $\mathcal{O}$ to the direction normal to $v$. Since in an unfolding, the parallel trajectories hit $\mathcal{R}_0(\mathcal{O})$, we must have
        \[\Omega = \pi(\mathcal{O}) = \pi (\mathcal{R}_0 \mathcal{O}) = \mathcal{R}_0 \Omega.\]
        The last equality holds since $\pi$ and $\mathcal{R}_0$ commute. This shows $\Omega$ is invariant under the rotation $\mathcal{R}_0$.\footnote{By a slight abuse of notation, here we identify $\mathcal{R}_0$ with the rotation in $\pi(\mathbb{R}^n)$. This is possible as $\mathcal{R}_0$ fixes $v$.} If we parametrise the tube $T$ in an unfolding by $\Omega \times \mathbb{R}$, we also notice that a point $(x, t) \in \Omega \times \mathbb{R}$ corresponds to the same point in $P$ as does the point $(\mathcal{R}_0x, t + L)$. If we think of doubles, this means that $T$ is given by the local isometry $F: \Omega \times \mathbb{R} \to D_0$ such that $F(\mathcal{R}_0x, t + L) = F(x, t)$ for any $(x, t)$.
        
        In the case when $k$ is odd, we may repeat the previous argument for $\mathcal{R}_0 \in O(n-1)$ that is orientation reversing. If $n = 2, 3$, this implies $\mathcal{R}_0$ is a reflection in a plane containing $v$ and note that upon applying $\mathcal{R}_0^2 = \id$, we get back to the original position, so every individual trajectory in the tube is periodic.
        
        In the case where $P$ is rational, we notice that the rotation $\mathcal{R}_0$ in the above argument is of finite order. Thus we may assume that each trajectory is periodic at the cost of multiplying $k$ by an integer $q$. We have $\Omega$ polyhedral, since it is an intersection of a finite number of faces projected to the plane perpendicular to $v$.
    \end{proof}
    
    \begin{rem}\rm
        Observe that the above proof will not work when $P$ is not convex, as one can easily come up with examples (for example, consider a cube with a ``needle'' poking inward from the ceiling) of non-convex polyhedra where the set $X(w)$ is not simply-connected, which disallows the application of the Brouwer fixed point theorem.
    \end{rem}

    Now we give a proof of Theorem \ref{th:non-periodic}, proved in \cite{GKT}. Similarly to Theorem \ref{th:general-dynamics} above, the proof was written down for $n = 3$ in \cite{GKT} and we give a slightly different argument that avoids the use of ``the $N_\delta$ sets'' and is fully self-contained. The main novelty is in Lemma \ref{lemma:stronglipschitz}, that asserts that under a bi-Lipschitz type condition, a map on a compact set is uniformly recurrent, or in other words $C^0$-admissible in the sense of Proposition \ref{prop:iso_adm}. 
    
    \begin{proof}[Proof of Theorem \ref{th:non-periodic}]
        Assume the contrary, i.e. assume the trajectory $\gamma$ generated by $x$ is a positive distance $\varepsilon$ away from $\mathcal{S}$. Consider the maximal tube $T = T(x)$ corresponding to $X(w)$.
        
        Consider now an arbitrary limit point $f^{m_i} x \to x^* \in T\Gamma_1$ with $m_i \to \infty$. Since the trajectory of $x$ is at $\varepsilon$ distance to $\mathcal{S}$, so is the trajectory of $x^*$, and so we may consider the maximal tube $T(x^*)$ associated to $x^*$. Since $w$ is not periodic, eventually $f^{m_i}x$ is not parallel to $x^*$, as otherwise $f^{m_i}x$ would eventually generate the same tube as $T(x^*) = T(x)$ and thus $w$ is periodic, contradiction. We will show this to be absurd, thus the closure of the trajectory of $x$ must hit $\mathcal{S}$. Write $w^*$ for the word associated to $x^*$.
        
        Fix a $\delta > 0$ and define the (closed) set $X_\delta(w^*)$ to consist of points at a distance at least $\delta$ from $\partial X(w^*)$ inside $X(\omega^*)$. We introduce the forward orbit of $X_\delta(w^*)$ as follows
        \[Z_\delta := \cup_{i = 0}^\infty f^i(X_\delta(w^*)).\]
        We denote the closure of this forward orbit by $Y_\delta := \overline{Z_\delta}$.
        We claim that $Y_\delta \subset T\Gamma_1$ is invariant by $f$.
        
        To prove the claim, consider a sequence $f^{n_i}x_i \to y \in Y_\delta$. Note that $y$ is not tangent to a face, as otherwise this would contradict the assumption that trajectories starting in $\intt X(w^*)$ are a positive distance away from $\mathcal{S}$. Since $f^{n_i}x_i$ are at a distance at least $\delta$ from $\mathcal{S}$, we have $f(y)$ satisfies the same property and the continuity of $f$ on $T\Gamma_1$ proves the claim as $f^{n_{i+1}}x_i \to f(y)$.
        
        Next, we claim that there are $c, C > 0$, such that for any $m \geq 0$ we have for all $y_1, y_2 \in Y_\delta$
            \begin{equation}\label{eq:stronglipschitz}
                cd(y_1, y_2) \leq d(f^my_1, f^my_2) \leq C d(y_1, y_2).
            \end{equation}
        Here $d(\cdot, \cdot)$ is the natural distance on $T\Gamma_1$. The existence of such constants (independent of $m$), for $y_1, y_2 \in Z_\delta$ follows since the angle between the unfolded tube $T^\infty$ and the faces it hits is bounded from below, as otherwise a trajectory starting in $\intt X(w^*)$ would come arbitrarily close to $\mathcal{S}$, contradicting the definition of a tube. The full claim follows by density and $f$ satisfies this ``strong Lipschitz property''.
        
        Now we prove a statement similar to Proposition \ref{prop:iso_adm}.\footnote{Compared to \cite[Theorem 1.16]{F} and to the proof found in \cite{GKT}, the following lemma proves that \emph{all} points in $Y_\delta$ are uniformly recurrent at the same time. The key to this is that $f$ satisfies the strong property \eqref{eq:stronglipschitz} and the idea from the proof of Proposition \ref{prop:iso_adm}.}
        
        \begin{lemma}\label{lemma:stronglipschitz}
            For each $\eta > 0$, the set of $k \in \mathbb{N}_0$ such that $\dist_{C^0}(f^k|_{Y_\delta}, \id) < \eta$ is relatively dense.
        \end{lemma}
        \begin{proof}
            The proof is similar to the proof of Proposition \ref{prop:iso_adm} and we just sketch the main differences. The main idea is that instead of having $f$ an isometry, we will use the property \eqref{eq:stronglipschitz}. 
            
            Take $y \in Y_\delta$ and consider the orbit $\mathcal{O}_y = \{f^iy: i \geq 0\}$. By compactness, there are $n_1, \dotso, n_N \in \mathbb{N}_0$ such that $f^{n_1}y, \dotso, f^{n_N}y$ are $\eta$-dense in $\overline{\mathcal{O}_y}$. Then the set of $m \in \mathbb{N}_0$ such that $f^my$ is $\eta$-close to $y$ is relatively dense, as for each $m$ there is an $i$ such that, using \eqref{eq:stronglipschitz}
            \[\eta > d(f^my, f^{n_i}y) \geq cd(f^{m-n_i}y, y).\]
            If we set $U = B(y, \eta)$, by above we see $f^{m-n_i}y \in f^{m-n_i} U \cap B(y, \frac{\eta}{c})$. 
            
            Consider a cover of $Y_\delta$ by balls $U_i = B(x_i, \eta)$ for $i = 1, \dotso, N$ and apply the previous claim to the map $(f, \dotso, f): (Y_\delta)^N \to (Y_\delta)^N$ (it satisfies the property \eqref{eq:stronglipschitz} with the same constants) and the point $(x_1, \dotso, x_N)$. Thus there is a relatively dense set $\mathcal{T} \subset \mathbb{N}_0$ such that $B(x_i, \frac{\eta}{c}) \cap f^kU_i \neq \emptyset$ for every $k \in \mathcal{T}$ and $i = 1, \dotso, N$.
            
            Take $y \in Y_\delta$, $k \in \mathcal{T}$ and choose $y_i \in B(x_i, \frac{\eta}{c}) \cap f^kU_i$. There is an $i$ such that $y \in U_i$ and so 
            \[d(y, y_i) < d(y, x_i) + d(x_i, y_i) < \eta (1 + \frac{1}{c}).\] 
            Since $f^ky \in f^kU_i$ and using \eqref{eq:stronglipschitz}, we have $d(f^ky, y_i) < 2C\eta$. By triangle inequality 
            \[d(y, f^ky) < d(y, y_i) + d(y_i, f^ky) < \eta (1 + \frac{1}{c} + 2C).\]
            This finishes the proof.
        \end{proof}

        By the proof of Proposition \ref{prop:maximaltube}, $X(w^*)$ is convex and every trajectory generated by a point on $\partial X(w^*)$ comes arbitrarily close to the singular set. Take $\delta = \varepsilon/4$ and pick $\varepsilon/4$-dense points $x_1, \dotso, x_N \in \partial X_\delta(w^*)$ such their trajectories come $\varepsilon/2$-close to the singular set in time (length) $L_0$. Denote by $T_\delta^\infty(f^{k}x^*)$ the unfolded tube generated by $X_\delta(w(f^{k}x^*))$, i.e. the complement of the $\delta$-neighbourhood of $\partial T^\infty(f^kx^*)$. Take $\eta$ in Lemma \ref{lemma:stronglipschitz} small enough, such that $\dist_{C^0}(f^k|_{Y_\delta}, \id) < \eta$ for $k \in \mathcal{T}$ relatively dense, so that the following holds. For $k \in \mathcal{T}$ and each $i = 1, \dotso, N$, the points $f^kx_i$ are close enough to $x_i$, so that the trajectories generated by $f^kx_i$ come $\varepsilon/2$-close to the singular set in time $L_0$. Since $\mathcal{T}$ is relatively dense, this means that there is an $L > 0$ such that for each $i$, the trajectories generated by $x_i$ come $\varepsilon/2$-close to the singular set in every segment of length $L$.
        
        Finally, consider $T^\infty(f^{m_i}x)$, i.e. the maximal unfolded tube generated by $f^{m_i}x$; it has a minimal diameter $\varepsilon$ by the assumptions. Taking $i \to \infty$ and by the construction above, a segment of length $L$ of one the trajectories generated by $x_j$, together with its $\varepsilon/2$-neighbourhood, will eventually be contained in the interior of $T^\infty(f^{m_i}x)$. Thus the interior of $T^\infty(f^{m_i}x)$ will contain a singular point, contradicting the definition of maximal tubes.
    \end{proof}
    \section{On a control estimate}\label{app:B}
    
    
In this Appendix we prove a result on control and observability for an elliptic PDE on a product space, which has been generalised in Theorem \ref{thm:cont_per_cond} above. The result is proved in unpublished lecture notes of N. Burq (see the footnote of \cite[pp. 17]{BZ1}). 

To this end, let $(M_x, g_x)$ and $(M_y, g_y)$ be two compact Riemannian manifolds with Lipschitz boundary. We will consider the product manifold $(M_x \times M_y, g_x \oplus g_y) =: (M, g)$. Denote by $-\Delta_g$ the positive-definite Laplace-Beltrami operator.

Recall that we say a subset $A \subset M_x$ satisfies the \emph{geometric control condition} or just \emph{(GCC)} if every geodesic $\gamma$ in $M_x$ hits $A$ in finite time. 

\begin{theorem}\label{thm:product}
        Let $s \in \RR$ 
        and assume $u \in H_0^1(M)$ satisfies
        \begin{align}\label{eq:elleq}
            (- \Delta_g - s) u = f.
        \end{align}
        Now let $\omega \subset M_x$ be an open, non-empty set, satisfying (GCC). Then there exists a constant $C = C(M, g) > 0$ 
        such that the following observability estimate holds:
        \begin{align}\label{eq:mainineq1}
            \lVert{u}\rVert_{L^2(M)} \leq C(\lVert{f}\rVert_{H^{-1}_xL^2_y (M)} + \lVert{u|_{\omega \times M_y}}\rVert_{L^2(\omega \times M_y)}).
        \end{align}
\end{theorem}
\begin{proof}
    Consider the basis of orthonormal Dirichlet eigenfunctions $\{e_k\}_{k = 1}^\infty \subset L^2(M_y)$, such that $-\Delta_{g_y} e_k = \lambda_k e_k$, where $0 < \lambda_1 \leq \lambda_2 \leq \dotso \nearrow \infty$ are eigenvalues. We consider $u$ and $f$ satisfying \eqref{eq:elleq}, and consider the Fourier expansion into $e_k$:
    \begin{align*}
        u(x, y) = \sum_{k = 1}^\infty u_k(x) e_k(y), \quad f(x, y) = \sum_{k = 1}^\infty f_k(x) e_k(y).
    \end{align*}
    The first sum converges in $H^1(M)$ as $u \in H_0^1(M)$ and thus the second sum converges in $H^{-1}(M)$. Moreover, we have that $u_k \in H_0^1(M_x)$ and $f_k \in H^{-1}(M_x)$ as $u \in H_0^1(M)$, and we may show the following formulas for all $k \in \mathbb{N}$ hold
    \begin{align*}
        u_k(x) = \int_{M_y} e_k(y) u(x, y) d\vol_y,
        \quad f_k(x) = \int_{M_y} e_k(y) f(x, y) d\vol_y.
    \end{align*}
    Here $d\vol_y$ denotes the volume density on $M_y$. Then we have, by using \eqref{eq:elleq} and the eigenfunction condition:
    \begin{align*}
        (-\Delta_g - s) \sum_{k = 1}^\infty e_k u_k 
        = \sum_{k = 1}^\infty e_k(-\Delta_{g_x} - s + \lambda_k) u_k = \sum_{k = 1}^\infty e_k f_k.
    \end{align*}
    Therefore, we conclude for $k \in \mathbb{N}$
    \begin{align}\label{eq:fourierk}
        (-\Delta_{g_x} - (s - \lambda_k)) u_k(x) = f_k(x).
    \end{align}
    The remaining ingredient is the observability estimate on $(M_x, g_x)$ given in Lemma \ref{lemma:auxiliary}. The claim follows by applying the lemma to equation \eqref{eq:fourierk} for each $k \geq 1$
    \begin{align}
        \begin{split}
            \lVert{u}\rVert_{L^2(M)}^2 = \sum_{k = 1}^\infty \lVert{u_k}\rVert_{L^2(M_x)}^2 &\leq C \Big(\sum_{k = 1}^\infty \lVert{f_k}\rVert^2_{H^{-1}(M_x)} + \sum_{k = 1}^\infty \lVert{u_k|_{\omega}}\rVert_{L^2(\omega)}^2\Big)\\
            &= C\big(\lVert{f}\rVert_{H^{-1}_x L^2_y (M)}^2 + \lVert{u|_{\omega \times M_y}}\rVert_{L^2(\omega \times M_y)}^2\big).
        \end{split}
    \end{align}
    \end{proof}
    
    We can actually handle periodic conditions similarly, so we have
\begin{theorem}\label{thm:controlperiodic}
    Assume the same setup as in the previous theorem, just with $M_y = [0, 1]$. This time we take periodic boundary conditions in \eqref{eq:elleq}, i.e. 
    \begin{align*}
        u(x, 0) = u(x, 1) \quad \mathrm{for} \quad x \in M_x, \quad u|_{\partial M_x \times [0, 1]} = 0.
    \end{align*}
    Then the estimate \eqref{eq:mainineq} holds.
\end{theorem}
\begin{proof}
    The proof is completely analogous to the proof of the previous theorem, by just noting that we may look at $u$ as a function on $M_x \times S^1$. Then we expand $u$ into Fourier basis on $S^1$ and apply Lemma \ref{lemma:auxiliary}.
\end{proof}

\end{document}